\documentclass[a4paper,12pt]{article}
\usepackage{amsmath}
\usepackage{amssymb}
\usepackage{amsfonts}
\usepackage{amsthm}
\usepackage{bm}
\usepackage{braket}
\usepackage[truedimen,top=45truemm,bottom=45truemm,left=30truemm,right=30truemm]{geometry}

\newtheorem{thm}{Theorem}
\newtheorem{prop}{Proposition}
\newtheorem{lem}{Lemma}
\newtheorem{cor}{Corollary}
\newtheorem{rem}{Remark}
\newtheorem{ex}{Example}
\newtheorem{defn}{Definition}

\renewcommand{\labelenumi}{\rm(\roman{enumi})}

\newcommand{\N}{\mathbb{N}}
\newcommand{\Z}{\mathbb{Z}}
\newcommand{\Q}{\mathbb{Q}}
\newcommand{\R}{\mathbb{R}}
\newcommand{\C}{\mathbb{C}}
\newcommand{\Qbar}{\overline{\mathbb{Q}}}
\newcommand{\Cbar}{\overline{C}}
\newcommand{\z}{\bm{z}}
\newcommand{\ba}{\bm{\alpha}}
\newcommand{\bg}{\bm{\gamma}}
\newcommand{\f}{\bm{f}}
\newcommand{\g}{\bm{g}}
\newcommand{\bb}{\bm{b}}
\newcommand{\e}{\bm{e}}
\newcommand{\bt}{\bm{t}}
\newcommand{\bl}{\bm{\lambda}}
\newcommand{\m}{\bm{\mu}}
\newcommand{\La}{\mathcal{L}}
\newcommand{\M}{\mathcal{M}}
\newcommand{\Eta}{\mathcal{H}}
\newcommand{\Nu}{\mathcal{N}}
\newcommand{\X}{\mathcal{X}}

\newcommand{\relmiddle}[1]{\mathrel{}\middle#1\mathrel{}}
\newcommand{\house}[1]{\begin{array}{|c|}\hline #1\end{array}}

\DeclareMathOperator*{\ord}{ord}
\DeclareMathOperator*{\den}{den}
\DeclareMathOperator*{\ind}{index}

\def\hsymb#1{\mbox{\strut\rlap{\smash{\Huge$#1$}}\quad}}
\def\diag{\mathop{\rm diag}\nolimits}

\begin{document}

\title{Algebraic independence of certain entire functions of two variables generated by linear recurrences}
\author{Haruki Ide}
\date{}
\maketitle

\begin{abstract}
In this paper we construct an entire function of two variables having the property that its values and its partial derivatives of any order at any distinct algebraic points are algebraically independent. Such an entire function is generated by a linear recurrence. In order to prove this result, we reduce the algebraic independency to that of Mahler functions of several variables by shifting the linear recurrence and apply the theory of Mahler functions. 
\end{abstract} 

\section{Introduction and the results}\label{sec:1}
In transcendental number theory, various authors have investigated necessary and sufficient conditions for the values of analytic functions at algebraic numbers to be algebraically independent. The earliest such result is the famous Lindemann-Weierstrass theorem, which asserts that the values $e^{\alpha_1},\ldots,e^{\alpha_n}$ of the exponential function at algebraic numbers $\alpha_1,\ldots,\alpha_n$ are algebraically independent if and only if $\alpha_1,\ldots,\alpha_n$ are linearly independent over the rationals (cf. Shidlovskii \cite{Shid}).

Some complex or $p$-adic entire functions are known to have the notable property that their values and their derivatives of any order at any nonzero distinct algebraic numbers are algebraically independent. As  the first such result, Nishioka established Theorem \ref{thm:factorial} below. Before stating the theorem, we introduce some notation used throughout this paper.

Let $p$ be $\infty$ or a prime number. Let $|\cdot|_p$ denote the usual absolute value or the standarized $p$-adic absolute value of the field $\Q$ of rational numbers according respectively as $p$ is $\infty$ or a prime number. We denote by $\Q_p$ the completion of $\Q$ with respect to $|\cdot|_p$, by $\Qbar_p$ the algebraic closure of $\Q_p$, and by $\C_p$ the completion of $\Qbar_p$. Note that $\Q_\infty$ is the field $\R$ of real numbers, and that $\Qbar_\infty$ and $\C_\infty$ are the field $\C$ of complex numbers. We also denote by $|\cdot|_p$ the absolute value of $\C_p$. Let $\Qbar$ denote the field of algebraic numbers, that is, the algebraic closure of $\Q$ in $\C$. We denote by $\Qbar^\times$ the set of nonzero algebraic numbers. For each prime number $p$, we fix an embedding of $\Qbar$ into $\C_p$. We denote by $f^{(l)}(x)$ the derivative of $f(x)$ of order $l$.

\begin{thm}[Nishioka \cite{N1986}]\label{thm:factorial}
Let $p$ be $\infty$ or a prime number. Let $a$ be an algebraic number with $0<|a|_p<1$. Define $f(x)=\sum_{k=0}^\infty a^{k!}x^k$. Then the infinite subset $\{f^{(l)}(\alpha)\mid\alpha\in\Qbar^\times,\ l\geq0\}$ of $\Qbar_p$ is algebraically independent over $\Q$.
\end{thm}

First we consider the case where $p$ is $\infty$. Fix an algebraic number $a$ with $0<|a|_\infty<1$ in what follows. Nishioka also proved the following

\begin{thm}[Nishioka \cite{N1996}]\label{thm:geometric}
Let $d$ be an integer greater than $1$. Define $g(x)=\sum_{k=0}^\infty a^{d^k}x^k$. Then the infinite subset $\{g^{(l)}(\alpha)\mid\alpha\in\Qbar^\times,\ l\geq0\}$ of $\C$ is algebraically independent over $\Q$.
\end{thm}

Theorem \ref{thm:geometric} was proved by using the fact that the function $g(x;z)=\sum_{k=0}^\infty x^kz^{d^k}$ satisfies the functional equation
\begin{equation}\label{eq:g(x;z)}
g(x;z)=xg(x;z^d)+z,
\end{equation}
which is essentially different from the situation of Theorem \ref{thm:factorial}. Mahler functions are analytic functions satisfying the functional equations such as \eqref{eq:g(x;z)} or those of more general forms. In order to prove Theorem \ref{thm:geometric} above, Nishioka \cite{N1996} established a criterion for the algebraic independence of the values of Mahler functions. Mahler's method, which treats the algebraic independence of the values of Mahler functions, has further applications as follows.

Let $\{R_k\}_{k\geq0}$ be a linear recurrence of nonnegative integers satisfying
\begin{equation}\label{eq:LRS}
R_{k+n}=c_1R_{k+n-1}+\cdots+c_nR_k\quad (k\geq 0),
\end{equation}
where $n\geq2$, $R_0,\ldots,R_{n-1}$ are not all zero, and $c_1,\ldots, c_n$ are nonnegative integers with $c_n\neq0$. We define a polynomial associated with \eqref{eq:LRS} by
\begin{equation}\label{eq:Phi(X)}
\Phi(X):=X^n-c_1X^{n-1}-\cdots-c_n.
\end{equation}
Define 
\begin{equation}\label{eq:F}
F(x):=\sum_{k=0}^\infty a^{R_k}x^k.
\end{equation}
The following theorem was proved by applying Nishioka's criterion.

\begin{thm}[Tanaka \cite{T1994}]\label{thm:F}
Suppose that $\Phi(\pm1)\neq0$ and that the ratio of any pair of distinct roots of $\Phi(X)$ is not a root of unity. Then the infinite subset $\{F^{(l)}(\alpha)\mid\alpha\in\Qbar^\times,\ l\geq0\}$ of $\C$ is algebraically independent over $\Q$.
\end{thm}

Using Mahler's method, Tanaka also constructed a complex entire function defined by an infinite product and having the property that its values and its derivatives of any order at any nonzero distinct algebraic numbers except its zeros are algebraically independent. Define
\begin{equation}\label{eq:G}
G(y):=\prod_{k=0}^\infty \left(1-a^{R_k}y\right).
\end{equation}
The following theorem was proved by applying another criterion, which was proved by Kubota \cite{Kubota} and improved by Nishioka \cite{N}.

\begin{thm}[Tanaka \cite{T2012}]\label{thm:G}
Suppose that $\Phi(\pm1)\neq0$, that the ratio of any pair of distinct roots of $\Phi(X)$ is not a root of unity, and that $\{R_k\}_{k\geq0}$ is not a geometric progression. Then the infinite subset $\{G^{(m)}(\beta)\mid\beta\in\Qbar^\times\setminus\{a^{-R_k}\}_{k\geq0},\ m\geq0\}$ of $\C$ is algebraically independent over $\Q$.
\end{thm}

\begin{rem}\label{rem:Rk}
{\rm 
It is shown in Remark 2 of Tanaka \cite{T1996} that, if $\Phi(\pm1)\neq0$ and if the ratio of any pair of distinct roots of $\Phi(X)$ is not a root of unity, then $R_k=c\rho^k+o(\rho^k)$, where $\rho>1$ and $c>0$, so that $F(x)$ and $G(y)$ are complex entire functions. 
}
\end{rem}

\begin{rem}\label{rem:geom}
{\rm 
In the case where $\{R_k\}_{k\geq0}$ is a geometric progression, Theorem \ref{thm:G} is not valid (cf. Tanaka \cite[Remark 2]{T2012}). Note that $\{R_k\}_{k\geq0}$ is a geometric progression if and only if $R_1\neq0$ and $R_kR_{k+2}=R_{k+1}^2$ for all $k$ with $0\leq k\leq n-2$. 
}
\end{rem}

One of the main purpose of this paper is to construct an entire function of two variables which possesses the notable algbraic independence property such as the functions stated in Theorems \ref{thm:factorial}--\ref{thm:G} even for its partial derivatives. In what follows, we consider not only the complex case but also the $p$-adic case. Again let $p$ be $\infty$ or a prime number and fix an algebraic number $a$ with $0<|a|_p<1$. Let $\{R_k\}_{k\geq0}$ be a linear recurrence of nonnegative integers satisfying \eqref{eq:LRS} and $\Phi(X)$ the polynomial defined by \eqref{eq:Phi(X)}. We note that the degree $n$ of $\Phi(X)$ is greater than $1$. In the case where $p$ is $\infty$, we suppose that $\{R_k\}_{k\geq0}$ satisfies the following condition, which is the same as that assumed in Theorem \ref{thm:G}:
\begin{enumerate}
\setlength{\leftskip}{4.5mm}
\renewcommand{\labelenumi}{(R)$_{\infty}$}
\item $\Phi(\pm1)\neq0$, the ratio of any pair of distinct roots of $\Phi(X)$ is not a root of unity, and $\{R_k\}_{k\geq0}$ is not a geometric progression.
\end{enumerate}
On the other hand, in the case where $p$ is a prime number, we suppose that $\{R_k\}_{k\geq0}$ satisfies the following condition, which is stronger than (R)$_\infty$ (cf. Tanaka \cite[Remark 1]{T1996}):
\begin{enumerate}
\setlength{\leftskip}{3mm}
\renewcommand{\labelenumi}{(R)$_p$}
\setcounter{enumi}{3}
\item  $\Phi(X)$ is irreducible over $\Q$ and the roots $\rho_1,\ldots,\rho_n$ of $\Phi(X)$ satisfy $\rho_1>\max\{|\rho_2|_\infty,\ldots,|\rho_n|_\infty\}$.
\end{enumerate}
Then $F(x)$ and $G(y)$ defined respectively by \eqref{eq:F} and \eqref{eq:G} are entire functions on $\C_p$. We define a two-variable function $\Theta(x,y)$, the main object in this paper, by
$$\Theta(x,y):=\sum_{k=0}^\infty a^{R_k}x^k\prod_{\substack{j=0\\j\neq k}}^\infty\left(1-a^{R_j}y\right).$$
By Remark \ref{rem:Rk}, $\Theta(x,y)$ is an entire function on $\C_p\times\C_p$. We assert that $F(x)$ and $-G'(y)$ are specializations of $\Theta(x,y)$. Indeed, substituting $y=0$ into $\Theta(x,y)$, we have $\Theta(x,0)=F(x)$, so that
\begin{equation}\label{eq:Theta and F}
\frac{\partial^{l}\Theta}{\partial x^l}(x,0)=F^{(l)}(x)\quad(l\geq0).
\end{equation}
On the other hand, substituting $x=1$ into $\Theta(x,y)$, we see by the logarithmic derivative of $G(y)$ that 
$$
\Theta(1,y)=\prod_{j=0}^\infty\left(1-a^{R_j}y\right)\times\sum_{k=0}^\infty \frac{a^{R_k}}{1-a^{R_k}y}=-G'(y),
$$
so that
\begin{equation}\label{eq:Theta and G'}
\frac{\partial^{m}\Theta}{\partial y^m}(1,y)=-G^{(m+1)}(y)\quad(m\geq0).
\end{equation}
To state our main theorem, let us introduce the following notation. For each algebraic number $\beta$, we define
$$N_\beta:=\sharp\{k\geq 0 \mid a^{-R_k}=\beta\}=\ord_{y=\beta}G(y).$$
Then, by Remark \ref{rem:Rk}, $N_\beta$ is $0$ or $1$ for all but finitely many $\beta$. The following theorem, which establishes the algebraic independence of the ``direct product'' of the infinite sets treated in Theorems \ref{thm:F} and \ref{thm:G}, is the main theorem of the present paper. 

\begin{thm}\label{thm:main}
Let $p$ be $\infty$ or a prime number. Suppose that $\{R_k\}_{k\geq0}$ satisfies the condition {\rm (R)}$_p$. Then the infinite subset 
$$
\left\{\frac{\partial^{l+m}\Theta}{\partial x^l\partial y^m}(\alpha,\beta)\relmiddle| \alpha\in\Qbar^\times,\ \beta\in\Qbar,\ l\geq0,\ m\geq N_\beta\right\}{\textstyle\bigcup}\left\{G^{(N_\beta)}(\beta)\relmiddle|\beta\in\Qbar^\times\right\}
$$
of $\Qbar_p$ is algebraically independent over $\Q$.
\end{thm}

By \eqref{eq:Theta and F}, \eqref{eq:Theta and G'}, and Theorem \ref{thm:main}, we can refine Theorems \ref{thm:F} and \ref{thm:G}, namely we obtain the algebraic independence of the union of the infinite sets treated in Theorems \ref{thm:F} and \ref{thm:G} as well as the nonzero derivatives at the zeros of the infinite product $G(y)$.

\begin{cor}\label{cor:union}
Let $p$ be $\infty$ or a prime number. Suppose that $\{R_k\}_{k\geq0}$ satisfies the condition {\rm (R)}$_p$. Then the infinite subset 
$$
\left\{F^{(l)}(\alpha)\relmiddle|\alpha\in\Qbar^\times, \ l\geq 0\right\}{\textstyle\bigcup}\left\{G^{(m)}(\beta)\relmiddle|\beta\in\Qbar^\times,\ m\geq N_\beta\right\}
$$
of $\Qbar_p$ is algebraically independent over $\Q$.
\end{cor}

Let us describe another corollary of Theorem \ref{thm:main}. We define
\begin{equation}\label{eq:Xi}
\Xi(x,y):=\frac{\partial\Theta}{\partial y}(x,y)=\prod_{j=0}^\infty\left(1-a^{R_j}y\right)\times\sum_{\substack{k_1,k_2\geq0\\k_1\neq k_2}}\frac{-a^{R_{k_1}+R_{k_2}}x^{k_1}}{(1-a^{R_{k_1}}y)(1-a^{R_{k_2}}y)}.
\end{equation}
Theorem \ref{thm:main} implies that, if $\{R_k\}_{k\geq0}$ is strictly increasing, then the entire function $\Xi(x,y)$ on $\C_p\times\C_p$ have the following notable property: The infinite set consisting of its values and its partial derivatives of any order at any distinct algebraic points $(\alpha,\beta)$ with $\alpha\neq0$ is algebraically independent.

\begin{cor}\label{cor:Xi}
Let $p$ be $\infty$ or a prime number. Suppose that $\{R_k\}_{k\geq0}$ satisfies the condition {\rm (R)}$_p$. Assume in addition that $\{R_k\}_{k\geq0}$ is strictly increasing. Then the infinite subset 
$$
\left\{\frac{\partial^{l+m}\Xi}{\partial x^l\partial y^m}(\alpha,\beta)\relmiddle| \alpha\in\Qbar^\times,\ \beta\in\Qbar,\ l\geq0,\ m\geq 0\right\}
$$
of $\Qbar_p$ is algebraically independent over $\Q$.
\end{cor}

\begin{ex}\label{ex:Fibonacci}
{\rm
Let $p$ be $\infty$ or a prime number and fix an algebraic number $a$ with $0<|a|_p<1$. Let $\{F_k\}_{k\geq0}$ be the Fibonacci numbers defined by
$$F_0=0,\quad F_1=1,\quad F_{k+2}=F_{k+1}+F_k\quad(k\geq0).$$
Regarding $\{F_{k+2}\}_{k\geq0}$ as $\{R_k\}_{k\geq0}$, we define the entire function $\Xi(x,y)$ on $\C_p\times\C_p$ by \eqref{eq:Xi}, namely,
$$\Xi(x,y)=\prod_{j=2}^\infty\left(1-a^{F_j}y\right)\times\sum_{\substack{k_1,k_2\geq2\\k_1\neq k_2}}\frac{-a^{F_{k_1}+F_{k_2}}x^{k_1-2}}{(1-a^{F_{k_1}}y)(1-a^{F_{k_2}}y)}.$$
Then by Corollary \ref{cor:Xi} the infinite subset
$$
\left\{\frac{\partial^{l+m}\Xi}{\partial x^l\partial y^m}(\alpha,\beta)\relmiddle| \alpha\in\Qbar^\times,\ \beta\in\Qbar,\ l\geq0,\ m\geq 0\right\}
$$
of $\Qbar_p$ is algebraically independent over $\Q$.
}
\end{ex}

Theorem \ref{thm:main} is deduced from Theorem \ref{thm:main2} below. We define
$$
H(x,y):=\sum_{k=0}^\infty\frac{a^{R_k}x^k}{1-a^{R_k}y}
$$
and
$$
F_m(x):=\frac{\partial^mH}{\partial y^m}(x,0)=m!\sum_{k=0}^\infty a^{(m+1)R_k}x^k\quad(m=0,1,2,\ldots).
$$
Then $H(x,y)$ is a holomorphic function on $\C_p\times(\C_p\setminus\{a^{-R_k}\}_{k\geq0})$ and $F_m(x)$ $(m\geq0)$ are entire functions on $\C_p$.

\begin{thm}\label{thm:main2}
Let $p$ be $\infty$ or a prime number. Suppose that $\{R_k\}_{k\geq0}$ satisfies the condition {\rm (R)}$_p$. Then the infinite subset 
\begin{gather*}
\left\{\frac{\partial^{l+m}H}{\partial x^l\partial y^m}(\alpha,\beta)\relmiddle| \alpha\in\Qbar^\times,\ \beta\in\Qbar^\times\setminus\{a^{-R_k}\}_{k\geq0},\ l\geq0,\ m\geq0\right\}\\
{\textstyle\bigcup}\left\{F_m^{(l)}(\alpha)\relmiddle| \alpha\in\Qbar^\times, \ l\geq 0,\ m\geq0\right\}{\textstyle\bigcup}\left\{G(\beta)\relmiddle|\beta\in\Qbar^\times\setminus\{a^{-R_k}\}_{k\geq0}\right\}
\end{gather*}
of $\Qbar_p$ is algebraically independent over $\Q$.
\end{thm}

This paper is organized as follows. In Section \ref{sec:2}, we reduce Theorem \ref{thm:main} to Theorem \ref{thm:main2} by shifting the linear recurrence $\{R_k\}_{k\geq0}$ so as to avoid the zeros of the infinite product $G(y)$. In Section \ref{sec:3}, we establish a criterion for the algebraic independence of the values of Mahler functions. Our criterion, which is valid not only in the complex case but also in the $p$-adic case, includes that of Nishioka and a special case of that of Kubota. In the last section, using our criterion, we reduce Theorem \ref{thm:main2} to the existence of nontrivial rational function solutions of certain types of functional equations and complete the proof by applying Tanaka's results.

\section{Proof of Theorem \ref{thm:main}}\label{sec:2}
In this section, we deduce Theorem \ref{thm:main} from Theorem \ref{thm:main2}.

\begin{proof}[Proof of Theorem \ref{thm:main}]
Since
$$G'(y)=\prod_{j=0}^\infty\left(1-a^{R_j}y\right)\times\sum_{k=0}^\infty \frac{-a^{R_k}}{1-a^{R_k}y}=G(y)(-H(1,y)),$$
we see inductively that, for any $m\geq1$,
$$G^{(m)}(y)=G(y)A_m\left(H(1,y),\ldots,\frac{\partial^{m-1}H}{\partial y^{m-1}}(1,y)\right),$$
where $A_m(X_1,\ldots,X_m)\in\Z[X_1,\ldots,X_m]$. Hence, by
$$\Theta(x,y)=\prod_{j=0}^\infty(1-a^{R_j}y)\times\sum_{k=0}^\infty\frac{a^{R_k}x^k}{1-a^{R_k}y}=G(y)H(x,y),$$
we have
\begin{align}\label{eq:Theta_H}
&\frac{\partial^{l+m}\Theta}{\partial x^l\partial y^m}(x,y)\nonumber\\
=\ &G(y)\frac{\partial^{l+m}H}{\partial x^l\partial y^m}(x,y)\nonumber\\
&+G(y)B_m\left(H(1,y),\ldots,\frac{\partial^{m-1}H}{\partial y^{m-1}}(1,y),\frac{\partial^l H}{\partial x^l}(x,y),\ldots,\frac{\partial^{l+m-1} H}{\partial x^l\partial y^{m-1}}(x,y)\right)
\end{align}
for any $l,m\geq0$, where $B_m(X_1,\ldots,X_m,Y_1,\ldots,Y_m)\in\Z[X_1,\ldots,X_m,Y_1,\ldots,Y_m]$. Then, for any $l,m\geq0$, we see that
\begin{align}\label{eq:H_Theta}
&\frac{\partial^{l+m}H}{\partial x^l\partial y^m}(x,y)\nonumber\\
=\ &\frac{1}{G(y)}\frac{\partial^{l+m}\Theta}{\partial x^l\partial y^m}(x,y)\nonumber\\
+&\ C_m\left(\frac{\Theta(1,y)}{G(y)},\ldots,\frac{1}{G(y)}\frac{\partial^{m-1}\Theta}{\partial y^{m-1}}(1,y),\frac{1}{G(y)}\frac{\partial^l\Theta}{\partial x^l}(x,y),\ldots,\frac{1}{G(y)}\frac{\partial^{l+m-1} \Theta}{\partial x^l\partial y^{m-1}}(x,y)\right),
\end{align}
where $C_m(X_1,\ldots,X_m,Y_1,\ldots,Y_m)\in\Z[X_1,\ldots,X_m,Y_1,\ldots,Y_m]$. In particular, substituting $y=0$ into both sides of \eqref{eq:Theta_H} and \eqref{eq:H_Theta}, we see that, for any $l,m\geq0$,
\begin{equation}\label{eq:Theta_F_m}
\frac{\partial^{l+m}\Theta}{\partial x^l\partial y^m}(x,0)=F_{m}^{(l)}(x)+B_m\left(F_0(1),\ldots,F_{m-1}(1),F_0^{(l)}(x),\ldots,F_{m-1}^{(l)}(x)\right)
\end{equation}
and
\begin{align}\label{eq:F_m_Theta}
&F_{m}^{(l)}(x)\nonumber\\
=\ &\frac{\partial^{l+m}\Theta}{\partial x^l\partial y^m}(x,0)+C_m\left(\Theta(1,0),\ldots,\frac{\partial^{m-1}\Theta}{\partial y^{m-1}}(1,0),\frac{\partial^l\Theta}{\partial x^l}(x,0),\ldots,\frac{\partial^{l+m-1} \Theta}{\partial x^l\partial y^{m-1}}(x,0)\right),
\end{align}
respectively.

Let $\alpha_1,\ldots,\alpha_r$ be any nonzero distinct algebraic numbers with $\alpha_1=1$ and $\beta_1,\ldots,\beta_s$ any nonzero distinct algebraic numbers. To simplify our notation, we denote $N_j:=N_{\beta_j}$ $(1\leq j\leq s)$. In order to prove the theorem, it is enough to prove that, for any sufficiently large $L$ and $M$, the finite set
\begin{align*}
T_1:=&\left\{\frac{\partial^{l+m}\Theta}{\partial x^l\partial y^m}(\alpha_i,\beta_j)\relmiddle| 1\leq i\leq r,\ 1\leq j\leq s,\ 0\leq l\leq L,\ N_j\leq m\leq N_j+M\right\}\\
&\bigcup\left\{\frac{\partial^{l+m}\Theta}{\partial x^l\partial y^m}(\alpha_i,0)\relmiddle| 1\leq i\leq r,\ 0\leq l\leq L,\ 0\leq m\leq M\right\}\\
&{\textstyle\bigcup}\left\{G^{(N_j)}(\beta_j)\relmiddle|1\leq j\leq s\right\}
\end{align*}
is algebraically independent over $\Q$. Following several steps, we reduce the algebraic independency of $T_1$ to that of another set. We see by \eqref{eq:Theta_F_m} and \eqref{eq:F_m_Theta} that the algebraic independency of $T_1$ is equivalent to that of
\begin{align*}
T_2:=&\left\{\frac{\partial^{l+m}\Theta}{\partial x^l\partial y^m}(\alpha_i,\beta_j)\relmiddle| 1\leq i\leq r,\ 1\leq j\leq s,\ 0\leq l\leq L,\ N_j\leq m\leq N_j+M\right\}\\
&{\textstyle\bigcup}\left\{F_m^{(l)}(\alpha_i)\relmiddle| 1\leq i\leq r,\ 0\leq l\leq L,\ 0\leq m\leq M\right\}\\
&{\textstyle\bigcup}\left\{G^{(N_j)}(\beta_j)\relmiddle|1\leq j\leq s\right\}.
\end{align*}
Then, since
$$\frac{\partial^m\Theta}{\partial y^m}(1,\beta_j)=-G^{(m+1)}(\beta_j)\quad(1\leq j\leq s,\ N_j\leq m\leq N_j+M)$$
by \eqref{eq:Theta and G'} in Section \ref{sec:1}, we see that the algebraic independency of $T_2$ is equivalent to that of
\begin{align*}
T_3:=&\left\{\frac{\partial^{l+m}\Theta}{\partial x^l\partial y^m}(\alpha_i,\beta_j)\relmiddle|1\leq i\leq r,\ 1\leq j\leq s,\ l_0(i)\leq l\leq L,\ N_j\leq m\leq N_j+M\right\}\\
&{\textstyle\bigcup}\left\{F_m^{(l)}(\alpha_i)\relmiddle|1\leq i\leq r,\ 0\leq l\leq L,\ 0\leq m\leq M\right\}\\
&{\textstyle\bigcup}\left\{G^{(m)}(\beta_j)\relmiddle|1\leq j\leq s,\ N_j\leq m\leq N_j+M+1\right\},
\end{align*}
where
$$l_0(i):=\left\{
\begin{aligned}
&1\quad(i=1),\\
&0\quad(2\leq i\leq r).
\end{aligned}
\right.$$
Since $R_k\to\infty$ as $k$ tends to infinity, there exists a sufficiently large integer $k_0$ such that $1-a^{R_k}\beta_j\neq0$ $(1\leq j\leq s)$ for all $k\geq k_0$. We let $\widetilde{R}_k:=R_{k+k_0}$ $(k\geq0)$ and define the functions $\widetilde{\Theta}(x,y)$, $\widetilde{F}_m(x)$ $(0\leq m\leq M)$, and $\widetilde{G}(y)$ corresponding to the linear recurrence $\{\widetilde{R}_k\}_{k\geq0}$ by
$$\widetilde{\Theta}(x,y):=\sum_{k=0}^\infty a^{\widetilde{R}_k}x^k\prod_{\substack{j=0\\j\neq k}}^\infty(1-a^{\widetilde{R}_j}y),$$
$$\widetilde{F}_m(x):=m!\sum_{k=0}^\infty a^{(m+1)\widetilde{R}_k}x^k\quad(0\leq m\leq M),$$
and
$$\widetilde{G}(y):=\prod_{k=0}^\infty (1-a^{\widetilde{R}_k}y).$$

We assert that the algebraic independency of $T_3$ is equivalent to that of
\begin{align*}
S_1:=&\left\{\frac{\partial^{l+m}\widetilde{\Theta}}{\partial x^l\partial y^m}(\alpha_i,\beta_j)\relmiddle|1\leq i\leq r,\ 1\leq j\leq s,\ l_0(i)\leq l\leq L,\ 0\leq m\leq M\right\}\\
&{\textstyle\bigcup}\left\{\widetilde{F}_m^{(l)}(\alpha_i)\relmiddle|1\leq i\leq r,\ 0\leq l\leq L,\ 0\leq m\leq M\right\}\\
&{\textstyle\bigcup}\left\{\widetilde{G}^{(m)}(\beta_j)\relmiddle|1\leq j\leq s,\ 0\leq m\leq M+1\right\}.
\end{align*}
To see this, we prove that $T_3$ modulo $\Qbar$ and $S_1$ modulo $\Qbar$ generate the same $\Qbar$-vector space. First, we show that $F_m^{(l)}(\alpha_i)$ $(1\leq i\leq r,\ 0\leq l\leq L,\ 0\leq m\leq M)$ can be represented as linear combinations of $\widetilde{F}_m^{(l)}(\alpha_i)$ $(1\leq i\leq r,\ 0\leq l\leq L,\ 0\leq m\leq M)$ modulo $\Qbar$. Since
\begin{align*}
F_m(x)&=m!\sum_{k=k_0}^\infty a^{(m+1)R_k}x^k+m!\sum_{k=0}^{k_0-1} a^{(m+1)R_k}x^k\\
&=R(x)\widetilde{F}_m(x)+m!\sum_{k=0}^{k_0-1} a^{(m+1)R_k}x^k\quad(0\leq m\leq M),
\end{align*}
where $R(x):=x^{k_0}$, we have 
$$F_m^{(l)}(\alpha_i)\equiv\sum_{h=0}^{l}\binom lh R^{(l-h)}(\alpha_i)\widetilde{F}_m^{(h)}(\alpha_i)\pmod\Qbar$$
for $1\leq i\leq r$, $0\leq l\leq L$, and $0\leq m\leq M$. Hence we find that
$$
\left(
\begin{array}{c}
F_m(\alpha_i) \\
F_m'(\alpha_i) \\
\vdots \\
F_m^{(L)}(\alpha_i)
\end{array}
\right)\equiv\left(
\begin{array}{cccc}
\alpha_i^{k_0}&&&\\
&\alpha_i^{k_0}&&\hsymb{0}\\
&&\ddots&\\
\hsymb{*}&&&\alpha_i^{k_0}
\end{array}
\right)\left(
\begin{array}{c}
\widetilde{F}_m(\alpha_i) \\
\widetilde{F}_m'(\alpha_i) \\
\vdots \\
\widetilde{F}_m^{(L)}(\alpha_i)
\end{array}
\right)\pmod{\Qbar^{L+1}}
$$
for $1\leq i\leq r$ and $0\leq m\leq M$. Letting
\begin{align*}
\f_{im}&:={}^t(F_m(\alpha_i),F_m'(\alpha_i),\ldots,F_m^{(L)}(\alpha_i))\quad(1\leq i\leq r,\ 0\leq m\leq M),\\
\f&:={}^t({}^t\f_{10},\ldots,{}^t\f_{1M},{}^t\f_{20},\ldots,{}^t\f_{2M},\ldots,{}^t\f_{r0},\ldots,{}^t\f_{rM}),\\
\widetilde{\f}_{im}&:={}^t(\widetilde{F}_m(\alpha_i),\widetilde{F}_m'(\alpha_i),\ldots,\widetilde{F}_m^{(L)}(\alpha_i))\quad(1\leq i\leq r,\ 0\leq m\leq M),\\
\widetilde{\f}&:={}^t({}^t\widetilde{\f}_{10},\ldots,{}^t\widetilde{\f}_{1M},{}^t\widetilde{\f}_{20},\ldots,{}^t\widetilde{\f}_{2M},\ldots,{}^t\widetilde{\f}_{r0},\ldots,{}^t\widetilde{\f}_{rM}),
\end{align*}
$$
A_i:=\left(
\begin{array}{cccc}
\alpha_i^{k_0}&&&\\
&\alpha_i^{k_0}&&\hsymb{0}\\
&&\ddots&\\
\hsymb{*}&&&\alpha_i^{k_0}
\end{array}
\right)\in GL_{L+1}(\Qbar)\quad(1\leq i\leq r),
$$
and
$$
A:=\diag(\underbrace{A_1,\ldots,A_1}_{M+1},\underbrace{A_2,\ldots,A_2}_{M+1},\ldots ,\underbrace{A_r,\ldots,A_r}_{M+1})\in GL_{r(L+1)(M+1)}(\Qbar),
$$
we obtain
$$\f\equiv A\widetilde{\f}\pmod{\Qbar^{r(L+1)(M+1)}}.$$
Secondly, we show that $G^{(m)}(\beta_j)$ $(1\leq j\leq s,\ N_j\leq m\leq N_j+M+1)$ can be represented as linear combinations of $\widetilde{G}^{(m)}(\beta_j)$ $(1\leq j\leq s,\ 0\leq m\leq M+1)$. For $1\leq j\leq s$, we define
$$P_j(y):=(1-\beta_j^{-1}y)^{N_j}\in\Qbar[y],\quad Q_j(y):=\prod_{\substack{k=0\\a^{R_k}\neq\beta_j^{-1}}}^{k_0-1} (1-a^{R_k}y)\in\Qbar[y].$$
Since
$$G(y)=\prod_{k=0}^{k_0-1}(1-a^{R_k}y)\times\prod_{k=k_0}^\infty(1-a^{R_k}y)=P_j(y)Q_j(y)\widetilde{G}(y)\quad(1\leq j\leq s),$$
we see that, for $1\leq j\leq s$ and $0\leq m\leq M+1$,
$$G^{(N_j+m)}(\beta_j)=\sum_{h=0}^{m}\binom {N_j+m}{N_j\quad m-h\quad h} p_jQ_j^{(m-h)}(\beta_j)\widetilde{G}^{(h)}(\beta_j),$$
where $p_j:=P_j^{(N_j)}(y)\in\Qbar^\times$. Then, noting that $q_j:=Q_j(\beta_j)\in\Qbar^\times$ $(1\leq j\leq s)$, we have
$$
\left(
\begin{array}{c}
G^{(N_j)}(\beta_j) \\
G^{(N_j+1)}(\beta_j) \\
\vdots \\
G^{(N_j+M+1)}(\beta_j)
\end{array}
\right)=\left(
\begin{array}{cccc}
p_jq_j&&&\\
&\binom{N_j+1}{N_j}p_jq_j&&\hsymb{0}\\
&&\ddots&\\
\hsymb{*}&&&\binom{N_j+M+1}{N_j}p_jq_j
\end{array}
\right)\left(
\begin{array}{c}
\widetilde{G}(\beta_j) \\
\widetilde{G}'(\beta_j) \\
\vdots \\
\widetilde{G}^{(M+1)}(\beta_j)
\end{array}
\right)
$$
for $1\leq j\leq s$. Letting
\begin{align*}
\g_j&:={}^t(G^{(N_j)}(\beta_j),G^{(N_j+1)}(\beta_j),\ldots,G^{(N_j+M+1)}(\beta_j))\quad(1\leq j\leq s),\\
\g&:={}^t({}^t\g_1,\ldots,{}^t\g_s),\\
\widetilde{\g}_j&:={}^t(\widetilde{G}(\beta_j),\widetilde{G}'(\beta_j),\ldots,\widetilde{G}^{(M+1)}(\beta_j))\quad(1\leq j\leq s),\\
\widetilde{\g}&:={}^t({}^t\widetilde{\g}_1,\ldots,{}^t\widetilde{\g}_s),
\end{align*}
$$
B_j:=\left(
\begin{array}{cccc}
p_jq_j&&&\\
&\binom{N_j+1}{N_j}p_jq_j&&\hsymb{0}\\
&&\ddots&\\
\hsymb{*}&&&\binom{N_j+M+1}{N_j}p_jq_j
\end{array}
\right)\in GL_{M+2}(\Qbar)\quad(1\leq j\leq s),
$$
and
$$B:=\diag(B_1,\ldots,B_s)\in GL_{s(M+2)}(\Qbar),$$
we obtain
$$\g=B\widetilde{\g}.$$
Finally, we show that $\partial^{l+m}\Theta/\partial x^l\partial y^m(\alpha_i,\beta_j)$ $(1\leq i\leq r,\ 1\leq j\leq s,\ l_0(i)\leq l\leq L,\ N_j\leq m\leq N_j+M)$ can be represented as linear combinations of $\partial^{l+m}\widetilde{\Theta}/\partial x^l\partial y^m$ $(\alpha_i,\beta_j)$ $(1\leq i\leq r,\ 1\leq j\leq s,\ l_0(i)\leq l\leq L,\ 0\leq m\leq M)$ and $\widetilde{G}^{(m)}(\beta_j)$ $(1\leq j\leq s,\ 0\leq m\leq M+1)$. For $1\leq j\leq s$, we define
$$U_j(y):=(1-\beta_j^{-1}y)^{\max\{N_j-1,0\}}\in\Qbar[y]$$
and
$$V_j(x,y):=\left(\sum_{k=0}^{k_0-1} a^{R_k}x^k\prod_{\substack{j'=0\\j'\neq k}}^{k_0-1}(1-a^{R_{j'}}y)\right)/U_j(y)\in\Qbar[x,y].$$
It is easy to check that
$$\Theta(x,y)=R(x)P_j(y)Q_j(y)\widetilde{\Theta}(x,y)+U_j(y)V_j(x,y)\widetilde{G}(y)\quad(1\leq j\leq s).$$
Using the fact that $N_j-\max\{N_j-1,0\}=\min\{1,N_j\}\ (1\leq j\leq s)$, we have
\begin{align}\label{eq:liner_comb}
&\frac{\partial^{l+N_j+m}\Theta}{\partial x^l\partial y^{N_j+m}}(\alpha_i,\beta_j)\nonumber\\
=\ &\sum_{h_1=0}^l\binom{l}{h_1}R^{(l-h_1)}(\alpha_i)\sum_{h_2=0}^m\binom{N_j+m}{N_j\quad m-h_2\quad h_2}p_jQ_j^{(m-h_2)}(\beta_j)\frac{\partial^{h_1+h_2}\widetilde{\Theta}}{\partial x^{h_1}\partial y^{h_2}}(\alpha_i,\beta_j)\nonumber\\
&+\sum_{h_3=0}^{m+\min\{1,N_j\}}\left(\binom{N_j+m}{\max\{N_j-1,0\}\quad m+\min\{1,N_j\}-h_3\quad h_3}\right.\nonumber\\
&\left.\quad\quad\quad\quad\quad\quad\quad\quad\times u_j\frac{\partial^{l+m+\min\{1,N_j\}-h_3}V_j}{\partial x^l\partial y^{m+\min\{1,N_j\}-h_3}}(\alpha_i,\beta_j)\widetilde{G}^{(h_3)}(\beta_j)\right)
\end{align}
for $1\leq i\leq r$, $1\leq j\leq s$, $l_0(i)\leq l\leq L,$ and $0\leq m\leq M$, where $ u_j:=U_j^{(\max\{N_j-1,0\})}(y)\in\Qbar^\times$. In particular, when $i=1$, from \eqref{eq:Theta and G'} in Section \ref{sec:1}, \eqref{eq:liner_comb}, and the fact that $\alpha_1=1$, we have
\begin{align}\label{eq:liner_comb2}
&\frac{\partial^{l+N_j+m}\Theta}{\partial x^l\partial y^{N_j+m}}(1,\beta_j)\nonumber\\
=\ &\sum_{h_1=1}^l\binom{l}{h_1}R^{(l-h_1)}(1)\sum_{h_2=0}^m\binom{N_j+m}{N_j\quad m-h_2\quad h_2}p_jQ_j^{(m-h_2)}(\beta_j)\frac{\partial^{h_1+h_2}\widetilde{\Theta}}{\partial x^{h_1}\partial y^{h_2}}(1,\beta_j)\nonumber\\
&+\sum_{h_3=0}^{m+\min\{1,N_j\}}\left(\binom{N_j+m}{\max\{N_j-1,0\}\quad m+\min\{1,N_j\}-h_3\quad h_3}\right.\nonumber\\
&\left.\quad\quad\quad\quad\quad\quad\quad\quad\times u_j\frac{\partial^{l+m+\min\{1,N_j\}-h_3}V_j}{\partial x^l\partial y^{m+\min\{1,N_j\}-h_3}}(1,\beta_j)\widetilde{G}^{(h_3)}(\beta_j)\right)\nonumber\\
&-R^{(l)}(1)\sum_{h_4=0}^m\binom{N_j+m}{N_j\quad m-h_4\quad h_4}p_jQ_j^{(m-h_4)}(\beta_j)\widetilde{G}^{(h_4+1)}(\beta_j)
\end{align}
for $1\leq j\leq s$, $1\leq l\leq L,$ and $0\leq m\leq M$. Let
\begin{align*}
\bm{\theta}_{ijl}&:=\sideset{^t}{}{\mathop{\left(\frac{\partial^{l+N_j}\Theta}{\partial x^l\partial y^{N_j}}(\alpha_i,\beta_j),\frac{\partial^{l+N_j+1}\Theta}{\partial x^l\partial y^{N_j+1}}(\alpha_i,\beta_j),\ldots,\frac{\partial^{l+N_j+M}\Theta}{\partial x^l\partial y^{N_j+M}}(\alpha_i,\beta_j)\right)}},\\
\widetilde{\bm{\theta}}_{ijl}&:=\sideset{^t}{}{\mathop{\left(\frac{\partial^l\widetilde{\Theta}}{\partial x^l}(\alpha_i,\beta_j),\frac{\partial^{l+1}\widetilde{\Theta}}{\partial x^l\partial y}(\alpha_i,\beta_j),\ldots,\frac{\partial^{l+M}\widetilde{\Theta}}{\partial x^l\partial y^M}(\alpha_i,\beta_j)\right)}}\\
&\quad\quad\quad\quad\quad\quad\quad\quad\quad\quad\quad\quad\quad\quad(1\leq i\leq r,\ 1\leq j\leq s,\ l_0(i)\leq l\leq L)
\end{align*}
and let
\begin{align*}
\bm{\theta}_{ij}:={}^t({}^t\bm{\theta}_{ij\,l_0(i)},{}^t\bm{\theta}_{ij\,l_0(i)+1},\ldots,{}^t\bm{\theta}_{ijL}),\quad\widetilde{\bm{\theta}}_{ij}:={}^t({}^t\widetilde{\bm{\theta}}_{ij\,l_0(i)},{}^t\widetilde{\bm{\theta}}_{ij\,l_0(i)+1},\ldots,{}^t\widetilde{\bm{\theta}}_{ijL})\\
\quad(1\leq i\leq r,\ 1\leq j\leq s).
\end{align*}
Then, from \eqref{eq:liner_comb} and \eqref{eq:liner_comb2}, we have
$$\bm{\theta}_{ijl}=\sum_{h=l_0(i)}^l\binom{l}{h}R^{(l-h)}(\alpha_i)E_j\widetilde{\bm{\theta}}_{ijh}+D_{ijl}\widetilde{\g}_j\quad(1\leq i\leq r,\ 1\leq j\leq s,\ l_0(i)\leq l\leq L),$$
where
$$
E_j:=\left(
\begin{array}{cccc}
p_jq_j&&&\\
&\binom{N_j+1}{N_j}p_jq_j&&\hsymb{0}\\
&&\ddots&\\
\hsymb{*}&&&\binom{N_j+M}{N_j}p_jq_j
\end{array}
\right)\in GL_{M+1}(\Qbar)
$$
and $D_{ijl}\in M_{M+1,M+2}(\Qbar)$, so that we have
$$\bm{\theta}_{ij}=C_{ij}\widetilde{\bm{\theta}}_{ij}+D_{ij}\widetilde{\g}_j\quad(1\leq i\leq r,\ 1\leq j\leq s),$$
where
$$
C_{ij}:=\left(
\begin{array}{cccc}
\alpha_i^{k_0}E_j&&&\\
&\alpha_i^{k_0}E_j&&\hsymb{0}\\
&&\ddots&\\
\hsymb{*}&&&\alpha_i^{k_0}E_j
\end{array}
\right)\in GL_{(L+1-l_0(i))(M+1)}(\Qbar)
$$
and 
$$D_{ij}:=\left(
\begin{array}{c}
D_{ij\,l_0(i)}\\
D_{ij\,l_0(i)+1}\\
\vdots\\
D_{ijL}
\end{array}
\right)\in M_{(L+1-l_0(i))(M+1),M+2}(\Qbar).$$
Letting
\begin{align*}
\bm{\theta}&:={}^t({}^t\bm{\theta}_{11},\ldots,{}^t\bm{\theta}_{1s},{}^t\bm{\theta}_{21},\ldots,{}^t\bm{\theta}_{2s},\ldots,{}^t\bm{\theta}_{r1},\ldots,{}^t\bm{\theta}_{rs}),\\
\widetilde{\bm{\theta}}&:={}^t({}^t\widetilde{\bm{\theta}}_{11},\ldots,{}^t\widetilde{\bm{\theta}}_{1s},{}^t\widetilde{\bm{\theta}}_{21},\ldots,{}^t\widetilde{\bm{\theta}}_{2s},\ldots,{}^t\widetilde{\bm{\theta}}_{r1},\ldots,{}^t\widetilde{\bm{\theta}}_{rs}),\\
C&:=\diag(C_{11},\ldots,C_{1s},C_{21},\ldots,C_{2s},\ldots,C_{r1},\ldots,C_{rs})\\
&\ \in GL_{sL(M+1)+(r-1)s(L+1)(M+1)}(\Qbar),
\end{align*}
and
$$D:=\left(
\begin{array}{c}
\diag(D_{11},\ldots,D_{1s})\\
\diag(D_{21},\ldots,D_{2s})\\
\vdots\\
\diag(D_{r1},\ldots,D_{rs})
\end{array}
\right)\in M_{sL(M+1)+(r-1)s(L+1)(M+1),s(M+2)}(\Qbar),
$$
we obtain
$$\bm{\theta}=C\widetilde{\bm{\theta}}+D\widetilde{\g}.$$
Therefore we have
$$
\left(
\begin{array}{c}
\f\\
\g\\
\bm{\theta}
\end{array}
\right)\equiv\left(
\begin{array}{ccc}
A&0&0\\
0&B&0\\
0&D&C
\end{array}
\right)\left(
\begin{array}{c}
\widetilde{\f}\\
\widetilde{\g}\\
\widetilde{\bm{\theta}}
\end{array}
\right)\pmod{\Qbar^N},
$$
where $N:=r(L+1)(M+1)+s(M+2)+sL(M+1)+(r-1)s(L+1)(M+1)$. This implies the assertion since the coefficient matrix of the right-hand side is a lower triangular matrix with entries in $\Qbar$ whose diagonal entries are nonzero. 

Since
$$\widetilde{G}^{(m+1)}(\beta_j)=-\frac{\partial^m\widetilde{\Theta}}{\partial y^m}(1,\beta_j)\quad(1\leq j\leq s,\ 0\leq m\leq M)$$
by \eqref{eq:Theta and G'} in Section \ref{sec:1}, we see that the algebraic independency of $S_1$ is equivalent to that of
\begin{align*}
S_2:=&\left\{\frac{\partial^{l+m}\widetilde{\Theta}}{\partial x^l\partial y^m}(\alpha_i,\beta_j)\relmiddle|1\leq i\leq r,\ 1\leq j\leq s,\ 0\leq l\leq L,\ 0\leq m\leq M\right\}\\
&{\textstyle\bigcup}\left\{\widetilde{F}_m^{(l)}(\alpha_i)\relmiddle|1\leq i\leq r,\ 0\leq l\leq L,\ 0\leq m\leq M\right\}\\
&{\textstyle\bigcup}\left\{\widetilde{G}(\beta_j)\relmiddle|1\leq j\leq s\right\}.
\end{align*}
Then we see by \eqref{eq:Theta_H} and \eqref{eq:H_Theta} that the algebraic independency of $S_2$ is equivalent to that of
\begin{align*}
S_3:=&\left\{\frac{\partial^{l+m}\widetilde{H}}{\partial x^l\partial y^m}(\alpha_i,\beta_j)\relmiddle|1\leq i\leq r,\ 1\leq j\leq s,\ 0\leq l\leq L,\ 0\leq m\leq M\right\}\\
&{\textstyle\bigcup}\left\{\widetilde{F}_m^{(l)}(\alpha_i)\relmiddle|1\leq i\leq r,\ 0\leq l\leq L,\ 0\leq m\leq M\right\}\\
&{\textstyle\bigcup}\left\{\widetilde{G}(\beta_j)\relmiddle|1\leq j\leq s\right\},
\end{align*}
where
$$\widetilde{H}(x,y):=\sum_{k=0}^\infty\frac{a^{\widetilde{R}_k}x^k}{1-a^{\widetilde{R}_k}y}.$$
This concludes the proof since Theorem \ref{thm:main2} for the linear recurrence $\{\widetilde{R}_k\}_{k\geq0}$ asserts that $S_3$ is algebraically independent over $\Q$.
\end{proof}

\section{Mahler functions of several variables}\label{sec:3}
\subsection{Multiplicative transformation $\Omega$}\label{sec:3.1}
We denote by $F(z_1,\ldots,z_n)$ and by $F[[z_1,\ldots,z_n]]$ the field of rational functions and the ring of formal power series in the variables $z_1,\ldots,z_n$ with coefficients in a field $F$, respectively, and by $F^\times$ the multiplicative group of nonzero elements of $F$.

Let $p$ be $\infty$ or a prime number. Let $\Omega=(\omega_{ij})$ be an $n\times n$ matrix with nonnegative integer entries. Then the maximum $\rho$ of the archimedean absolute values of the eigenvalues of $\Omega$ is itself an eigenvalue of $\Omega$ (cf. Gantmacher \cite[p. 66]{Gant}). We define a multiplicative transformation $\Omega:\C_p^n\to\C_p^n$ by
\begin{equation}\label{eq:Omegaz}
\Omega\z:=\left(\prod_{j=1}^nz_j^{\omega_{1j}},\prod_{j=1}^nz_j^{\omega_{2j}},\ldots,\prod_{j=1}^nz_j^{\omega_{nj}}\right)
\end{equation}
for any $\z=(z_1,\ldots,z_n)\in\C_p^n$. Then the iterates $\Omega^k\z$ $(k=0,1,2,\ldots)$ are well-defined. Let $\ba=(\alpha_1,\ldots,\alpha_n)$ be a point with $\alpha_1,\ldots,\alpha_n$ nonzero algebraic numbers. We consider the following four conditions on $\Omega$ and $\ba$.
\begin{enumerate}
\setlength{\leftskip}{3mm}
\renewcommand{\labelenumi}{(\Roman{enumi})}
\item $\Omega$ is nonsingular and none of its eigenvalues is a root of unity, so that $\rho>1$.
\item Every entry of the matrix $\Omega^k$ is $O(\rho^k)$ as $k$ tends to infinity.
\renewcommand{\labelenumi}{(\Roman{enumi})$_p$}
\item If we put $\Omega^k\ba=(\alpha_1^{(k)},\ldots ,\alpha_n^{(k)})$, then
$$
\log|\alpha_i^{(k)}|_p\leq-c\rho^k\quad(1\leq i\leq n) 
$$
for all sufficiently large $k$, where $c$ is a positive constant.
\end{enumerate}
In the case where $p$ is $\infty$, the last condition is the following
\begin{enumerate} 
\renewcommand{\labelenumi}{(\Roman{enumi})$_\infty$}
\setlength{\leftskip}{4.5mm}
\setcounter{enumi}{3}
\item For any nonzero $f(\z)\in\C[[z_1,\ldots,z_n]]$ which converges in some neighborhood of the origin of $\C^n$, there are infinitely many positive integers $k$ such that $f(\Omega^k\ba)\neq0$.
\end{enumerate}
On the other hand, in the case where $p$ is a prime number, the last condition is the following
\begin{enumerate}
\setlength{\leftskip}{3mm}
\renewcommand{\labelenumi}{(\Roman{enumi})$_p$}
\setcounter{enumi}{3}
\item For any nonzero $f(\z)\in\C_p[[z_1,\ldots,z_n]]$ which converges in some neighborhood of the origin of $\C_p^n$ and for any positive integer $a$, there are infinitely many positive integers $k$ such that $f(\Omega^{ak}\ba)\neq0$. 
\end{enumerate}

\subsection{Vanishing theorems}\label{sec:3.2}
In the case where $p$ is $\infty$, the condition (IV)$_\infty$ stated above has been studied by Mahler, Kubota, Loxton and van der Poorten, and Masser. The following lemma is Masser's vanishing theorem.

\begin{lem}[Masser \cite{Masser}]\label{lem:Masser}
Let $p$ be $\infty$ and $\Omega$ an $n\times n$ matrix with nonnegative integer entries satisfying the condition {\rm(I)}. Let $\ba$ be an $n$-dimensional vector whose components $\alpha_1,\ldots,\alpha_n$ are nonzero algebraic numbers such that $\Omega^k\ba\to(0,\ldots,0)$ in $\C^n$ as $k$ tends to infinity. Then the negation of the condition {\rm(IV)}$_\infty$ is equivalent to the following: There exist integers $i_1,\ldots,i_n$, not all zero, and positive integers $a,b$ such that
$$
(\alpha_1^{(k)})^{i_1}\cdots(\alpha_n^{(k)})^{i_n}=1
$$
for all $k=a+lb$ $(l=0,1,2,\ldots)$.
\end{lem}

On the other hand, in the case where $p$ is a prime number, Masser's vanishing theorem is unsolved. However, the following lemma, which is the $p$-adic analogue of Mahler's vanishing theorem \cite{Mahler}, can be proved in the same way as in the proof of Theorem 2.2 in Nishioka \cite{N}.

\begin{lem}\label{lem:Mahler}
Let $p$ be a prime number and $\Omega$ an $n\times n$ matrix with nonnegative integer entries. Suppose that the characteristic polynomial of $\Omega$ is irreducible over $\Q$ and that $\Omega$ has an eigenvalue $\rho>1$ which is greater than the archimedean absolute values of any other eigenvalues. We denote by $A_{ij}$ the $(i,j)$-cofactor of the matrix $\Omega-\rho E$, where $E$ is the identity matrix. Then $A_{i1}\neq0$ $(1\leq i\leq n)$. Moreover, if nonzero algebraic numbers $\alpha_1,\ldots,\alpha_n$ satisfy
$$
\sum_{i=1}^n|A_{i1}|_\infty\log|\alpha_i|_p<0,
$$
then the matrix $\Omega$ and the point $\ba=(\alpha_1,\ldots,\alpha_n)$ satisfy the conditions {\rm(I)}, {\rm(II)}, {\rm(III)}$_p$, and {\rm(IV)}$_p$.
\end{lem}

\subsection{Criterion for algebraic independence}\label{sec:3.3}
Mahler functions of several variables are analytic functions which satisfy certain types of functional equations under the transformation $\z\mapsto\Omega\z$ defined by \eqref{eq:Omegaz}. Kubota \cite{Kubota} studied Mahler functions $g_1(\z),\ldots,g_m(\z)$ satisfying respective functional equations
$$
\left(
\begin{array}{c}
g_1(\z)\\
\vdots \\
g_m(\z)
\end{array}
\right)
=\left(
\begin{array}{ccc}
e_1(\z)&&0\\
&\ddots&\\
0&&e_m(\z)
\end{array}
\right)\left(
\begin{array}{c}
g_1(\Omega\z)\\
\vdots \\
g_m(\Omega\z)
\end{array}
\right)+\left(
\begin{array}{c}
b_1(\z)\\
\vdots \\
b_m(\z)
\end{array}
\right),
$$
where $e_h(\z)$, $b_h(\z)\in\Qbar(z_1,\ldots,z_n)$ $(1\leq h\leq m)$, and established a criterion for the algebraic independence of their values as well as that of the functions themselves (see also Nishioka \cite{N}). On the other hand, Nishioka \cite{N1996} studied Mahler functions $f_{ij}(\z)$ $(1\leq i\leq l,\ 1\leq j\leq n(i))$ satisfying a system of functional equations
$$
\left(
\begin{array}{c}
\f_1(\z)\\
\vdots \\
\f_l(\z)
\end{array}
\right)
=\left(
\begin{array}{ccc}
A_1&&0\\
&\ddots&\\
0&&A_l
\end{array}
\right)\left(
\begin{array}{c}
\f_1(\Omega\z)\\
\vdots \\
\f_l(\Omega\z)
\end{array}
\right)+\left(
\begin{array}{c}
\bb_1(\z)\\
\vdots \\
\bb_l(\z)
\end{array}
\right),
$$
where 
\begin{equation}\label{eq:f_i}
\f_i(\z)={}^t(f_{i1}(\z),\ldots,f_{in(i)}(\z))\quad(1\leq i\leq l),
\end{equation} 
\begin{equation}\label{eq:A_i}
A_i=\left(
\begin{array}{cccc}
a_i&&&\\
a_{21}^{(i)}&a_i&&\hsymb{0}\\
\vdots&&\ddots&\\
a_{n(i)1}^{(i)}&\cdots&a_{n(i)\,n(i)-1}^{(i)}&a_i
\end{array}
\right)\in GL_{n(i)}(\Qbar),\quad a_i\neq0,\quad a_{s\,s-1}^{(i)}\neq0,
\end{equation}
and 
\begin{equation}\label{eq:bb_i}
\bb_i(\z)={}^t(b_{i1}(\z),\ldots,b_{in(i)}(\z))\in\Qbar(z_1,\ldots,z_n)^{n(i)}\quad(1\leq i\leq l),
\end{equation} 
and established a criterion for the algebraic independence of their values as well as that of the functions themselves.

In order to prove Theorem \ref{thm:main2}, we need the following criterion for the algebraic independence of the values of Mahler functions, which includes Nishioka's and a special case of Kubota's criteria. In what follows, we call a subfield $K$ of $\Qbar$ a number field if $K$ is a finite extension of $\Q$.

\begin{thm}\label{thm:criterion}
Let $p$ be $\infty$ or a prime number, $K$ a number field, and $\Omega$ an $n\times n$ matrix with nonnegative integer entries. Let $f_{ij}(\z),\ g_h(\z)\in K[[z_1,\ldots,z_n]]$ $(1\leq i\leq l,\ 1\leq j\leq n(i),\ 1\leq h\leq m)$ with $g_h(\bm{0})\neq0$ $(1\leq h\leq m)$ converge in an $n$-polydisc $U$ around the origin of $\C_p^n$. Suppose that they satisfy the system of functional equations
$$
\left(
\begin{array}{c}
\f_1(\z)\\
\vdots \\
\f_l(\z)\\
g_1(\z)\\
\vdots \\
g_m(\z)
\end{array}
\right)
=\left(
\begin{array}{ccc|ccc}
A_1&&0&&&\\
&\ddots&&&\hsymb{0}&\\
0&&A_l&&&\\\hline
&&&e_1(\z)&&0\\
&\hsymb{0}&&&\ddots&\\
&&&0&&e_m(\z)
\end{array}
\right)\left(
\begin{array}{c}
\f_1(\Omega\z)\\
\vdots \\
\f_l(\Omega\z)\\
g_1(\Omega\z)\\
\vdots \\
g_m(\Omega\z)
\end{array}
\right)+\left(
\begin{array}{c}
\bb_1(\z)\\
\vdots \\
\bb_l(\z)\\
0\\
\vdots \\
0
\end{array}
\right),
$$
where $\f_i(\z)$, $A_i\in GL_{n(i)}(\Qbar)$, and $\bb_i(\z)\in\Qbar(z_1,\ldots,z_n)^{n(i)}$ $(1\leq i\leq l)$ are as in \eqref{eq:f_i}, \eqref{eq:A_i}, and \eqref{eq:bb_i}, respectively, and $e_h(\z)\in\Qbar(z_1,\ldots,z_n)$ $(1\leq h\leq m)$. Let $\ba=(\alpha_1,\ldots,\alpha_n)$ be a point in $U$ whose components are nonzero algebraic numbers. Assume that $\Omega$ and $\ba$ satisfy the conditions {\rm(I)}, {\rm(II)}, {\rm(III)}$_p$, and {\rm(IV)}$_p$. Assume further that $b_{ij}(\Omega^k\ba)$ $(1\leq i\leq l,\ 1\leq j\leq n(i))$ and $e_h(\Omega^k\ba)$ $(1\leq h\leq m)$ are defined and $e_h(\Omega^k\ba)\neq0$ $(1\leq h\leq m)$ for all $k\geq0$.

Then, if the numbers $f_{ij}(\ba)$ $(1\leq i\leq l,\ 1\leq j\leq n(i))$ and $g_h(\ba)$ $(1\leq h\leq m)$ of $\Qbar_p$ are algebraically dependent over $\Q$, then at least one of the following two conditions holds:
\begin{enumerate}
\item There exist a nonempty subset $\{i_1,\ldots,i_r\}$ of $\{1,\ldots,l\}$ and nonzero algebraic numbers $c_1,\ldots,c_r$ such that
$$a_{i_1}=\cdots=a_{i_r}$$
and
$$R(\z):=c_1f_{i_11}(\z)+\cdots+c_rf_{i_r1}(\z)\in\Qbar(z_1,\ldots,z_n).$$
Here $R(\z)$ satisfies the functional equation
$$R(\z)=a_{i_1}R(\Omega\z)+c_1b_{i_11}(\z)+\cdots +c_rb_{i_r1}(\z).$$
\item There exist integers $d_1,\ldots,d_m$, not all zero, and $S(\bm{z})\in\Qbar(z_1,\ldots,z_n)^\times$ such that
$$S(\z)=S(\Omega\z)\prod_{h=1}^me_h(\z)^{d_h}.$$
\end{enumerate}
\end{thm}

The proof consists of two parts. The first is Theorem \ref{thm:indep_functions} below, the algebraic independence over the field of rational functions of Mahler functions themselves, which can be obtained by combining the proof of Theorem 3 in Nishioka \cite{N1996} and the second half of that of Theorem 3.5 in Nishioka \cite{N}.

\begin{thm}\label{thm:indep_functions}
Let $C$ be a field of characteristic $0$ and $M$ the quotient field of\newline $C[[z_1,\ldots,z_n]]$. Let $\Omega$ be an $n\times n$ matrix with nonnegative integer entries satisfying the condition {\rm(I)}. Suppose that $f_{ij}(\z)\in M$ $(1\leq i\leq l,\ 1\leq j\leq n(i))$ satisfy the system of functional equations
$$\left(
\begin{array}{c}
f_{i1}(\Omega\z) \\
\vdots \\
\vdots \\
f_{in(i)}(\Omega\z)
\end{array}
\right)=\left(
\begin{array}{cccc}
a_i&&&\\
a_{21}^{(i)}&a_i&&\hsymb{0}\\
\vdots&&\ddots&\\
a_{n(i)1}^{(i)}&\cdots&a_{n(i)\,n(i)-1}^{(i)}&a_i
\end{array}
\right)\left(
\begin{array}{c}
f_{i1}(\z) \\
\vdots \\
\vdots \\
f_{in(i)}(\z)
\end{array}
\right)+\left(
\begin{array}{c}
b_{i1}(\z) \\
\vdots \\
\vdots \\
b_{in(i)}(\z)
\end{array}
\right),$$
where $a_i$, $a_{st}^{(i)}\in C$, $a_i\neq0$, $a_{s\,s-1}^{(i)}\neq0$, and $b_{ij}(\z)\in C(z_1,\ldots,z_n)$. Assume that $g_h(\z)\in M^\times$ $(1\leq h\leq m)$ satisfy the functional equations
$$g_h(\Omega\z)=e_h(\z)g_h(\z)\quad (1\leq h\leq m),$$
where $e_h(\z)\in C(z_1,\ldots,z_n)$ $(1\leq h\leq m)$. 

Then, if $f_{ij}(\z)$ $(1\leq i\leq l,\ 1\leq j\leq n(i))$ and $g_h(\z)$ $(1\leq h\leq m)$ are algebraically dependent over $C(z_1,\ldots,z_n)$, then at least one of the following two conditions holds:
\begin{enumerate}
\item There exist a nonempty subset $\{i_1,\ldots,i_r\}$ of $\{1,\ldots,l\}$ and nonzero elements $c_1,\ldots,c_r$ of $C$ such that
$$a_{i_1}=\cdots=a_{i_r}$$
and
$$R(\z):=c_1f_{i_11}(\z)+\cdots+c_rf_{i_r1}(\z)\in C(z_1,\ldots,z_n).$$
Here $R(\z)$ satisfies the functional equation
$$R(\Omega\z)=a_{i_1}R(\z)+c_1b_{i_11}(\z)+\cdots +c_rb_{i_r1}(\z).$$
\item There exist integers $d_1,\ldots,d_m$, not all zero, and $S(\bm{z})\in C(z_1,\ldots,z_n)^\times$ such that
$$S(\Omega\z)=S(\z)\prod_{h=1}^me_h(\z)^{d_h}.$$
\end{enumerate}
\end{thm}

The second part, Theorem \ref{thm:indep_values} below, asserts the algebraic independence of the values of Mahler functions under the assumption that the Mahler functions themselves are algebraically independent over the field of rational functions.

\begin{thm}\label{thm:indep_values}
Let $p$ be $\infty$ or a prime number, $K$ a number field, and $\Omega$ an $n\times n$ matrix with nonnegative integer entries. Let $f_i(\z),\ g_h(\z)\in K[[z_1,\ldots,z_n]]$ $(1\leq i\leq l,\ 1\leq h\leq m)$ with $g_h(\bm{0})\neq0$ $(1\leq h\leq m)$ converge in an $n$-polydisc $U$ around the origin of $\C_p^n$. Suppose that $f_i(\z)$ $(1\leq i\leq l)$ satisfy the system of functional equations
\begin{equation}\label{eq:f}
\left(
\begin{array}{c}
f_1(\z)\\
\vdots \\
f_l(\z)
\end{array}
\right)
=A\left(
\begin{array}{c}
f_1(\Omega\z)\\
\vdots \\
f_l(\Omega\z)
\end{array}
\right)+\left(
\begin{array}{c}
b_1(\z)\\
\vdots \\
b_l(\z)
\end{array}
\right),
\end{equation}
where $A$ is an $l\times l$ matrix with entries in $K$ and $b_i(\z)\in K(z_1,\ldots,z_n)$ $(1\leq i\leq l)$. Assume that $g_h(\z)$ $(1\leq h\leq m)$ satisfy the functional equations
\begin{equation}\label{eq:g}
g_h(\z)=e_h(\z)g_h(\Omega\z)\quad (1\leq h\leq m),
\end{equation}
where $e_h(\z)\in K(z_1,\ldots,z_n)$ $(1\leq h\leq m)$. Let $\ba=(\alpha_1,\ldots,\alpha_n)$ be a point in $U$ whose components are nonzero algebraic numbers. Suppose that $\Omega$ and $\ba$ satisfy the conditions {\rm(I)}, {\rm(II)}, {\rm(III)}$_p$, and {\rm(IV)}$_p$ and that $b_i(\Omega^k\ba), e_h(\Omega^k\ba)$ $(1\leq i\leq l,\ 1\leq h\leq m)$ are defined and $e_h(\Omega^k\ba)\neq0$ $(1\leq h\leq m)$ for all $k\geq0$. 

Then, if the functions $f_i(\z)$ $(1\leq i\leq l)$ and $g_h(\z)$ $(1\leq h\leq m)$ are algebraically independent over $K(z_1,\ldots,z_n)$, then the numbers $f_i(\ba)$ $(1\leq i\leq l)$ and $g_h(\ba)$ $(1\leq h\leq m)$ of $\Qbar_p$ are algebraically independent over $\Q$.
\end{thm}

We prove Theorem \ref{thm:indep_values} in the next subsection. Let us now introduce some notation which will be used in the proof of Theorem \ref{thm:indep_values}. For any algebraic number $\alpha$, we denote by $\house{\alpha}$ the maximum of the archimedean absolute values of the conjugates of $\alpha$ and by $\den(\alpha)$ the least positive integer $d$ such that $d\alpha$ is an algebraic integer. We define 
$$\|\alpha\|:=\max\{\house{\alpha}\,,\ \den(\alpha)\}.$$
It is easily seen that
$$\left\|\sum_{i=1}^n\alpha_i\right\|\leq n\prod_{i=1}^n\|\alpha_i\|$$
and
$$\left\|\prod_{i=1}^n\alpha_i\right\|\leq \prod_{i=1}^n\|\alpha_i\|$$
for any algebraic numbers $\alpha_1,\ldots,\alpha_n$. Furthermore, for any nonzero algebraic number $\alpha$, we have
$$\|\alpha^{-1}\|\leq\|\alpha\|^{2[\Q(\alpha):\Q]}$$
(cf. Nishioka \cite{N1996})
and the fundamental inequality
\begin{equation}\label{eq:Liouville}
|\alpha|_p\geq\|\alpha\|^{-2[\Q(\alpha):\Q]}
\end{equation}
(cf. Waldschmidt \cite{Wald}).

\subsection{Proof of Theorem \ref{thm:indep_values}}
We denote by $\N$ the set of nonnegative integers. If $\bl$ is a vector whose components are nonnegative integers, then we denote by $|\bl|$ the sum of the components of $\bl$. The following lemma plays a crucial role in the proof. 

\begin{lem}[Nishioka \cite{N1996}]\label{lem:crucial}
Let $p$ be $\infty$ or a prime number, $\Omega$ an $n\times n$ matrix with nonnegative integer entries, and $\ba$ an $n$-dimensional vector whose components $\alpha_1,\ldots,\alpha_n$ are nonzero algebraic numbers. Suppose that $\Omega$ and $\ba$ satisfy the conditions {\rm(I)}, {\rm(II)}, {\rm(III)}$_p$, and {\rm(IV)}$_p$. Define the function
$$\psi(\z;x)=\sum_{i=1}^q\sum_{j=1}^{d_i}x^{j-1}\gamma_i^xh_{ij}(\z),$$
where $\gamma_1,\ldots,\gamma_q$ are nonzero distinct elements of $\C_p$ and $h_{ij}(\z)\in\C_p[[z_1,\ldots,z_n]]$ $(1\leq i\leq q,\ 1\leq j\leq d_i)$ converge in an $n$-polydisc $U$ around the origin of $\C_p^n$. Then, if $\psi(\Omega^k\ba;k)=0$ for all sufficiently large $k$, then $h_{ij}(\z)=0$ for every $i,j$.
\end{lem}

This lemma was proved by Nishioka \cite{N1996} in the case where $p$ is $\infty$. The proof is also valid in the case where $p$ is a prime number.

\begin{proof}[Proof of Theorem \ref{thm:indep_values}]
We may assume that $\alpha_1,\ldots,\alpha_n$ and the eigenvalues of $A$ are all contained in $K$. Since $f_1(\z),\ldots,f_l(\z)$ are algebraically independent over $K(z_1,\ldots,z_n)$, we have $\det A\neq0$. We let $\f(\z):={}^t(f_1(\z),\ldots,f_l(\z))$, $\bb(\z):={}^t(b_1(\z),\ldots,b_l(\z))$, and $\g(\z):={}^t(g_1(\z),\ldots,g_m(\z))$. Iterating the functional equations \eqref{eq:f} and \eqref{eq:g}, we see that
\begin{equation}\label{eq:fk}
\f(\z)=A^k\f(\Omega^k\z)+\bb^{(k)}(\z)\quad(k\geq0)
\end{equation}
and
\begin{equation}\label{eq:gk}
g_h(\z)=e_h^{(k)}(\z)g_h(\Omega^k\z)\quad(1\leq h\leq m,\ k\geq0),
\end{equation}
where
\begin{equation}\label{eq:bk}
\bb^{(k)}(\z)={}^t(b_1^{(k)}(\z),\ldots,b_l^{(k)}(\z)):=\sum_{j=0}^{k-1}A^j\bb(\Omega^j\z)\in K(z_1,\ldots,z_n)^l
\end{equation}
and
\begin{equation}\label{eq:ek}
e_h^{(k)}(\z):=\prod_{j=0}^{k-1}e_h(\Omega^j\z)\in K(z_1,\ldots,z_n).
\end{equation}
We note here that, any power of $\Omega$ and the point $\ba$ also satisfy the conditions (I), (II), (III)$_p$, and (IV)$_p$. Indeed, it is clear that they satisfy the conditions (I), (II), and (III)$_p$. If $p$ is $\infty$, then we see by Lemma \ref{lem:Masser} that they satisfy the condition (IV)$_\infty$, and if $p$ is a prime number, then it is obvious that they satisfy the condition (IV)$_p$. Therefore, taking a sufficiently large integer $k_0$ and replacing $\Omega$, $A$, $b_i(\z)$, and $e_h(\z)$ with $\Omega^{k_0}$, $A^{k_0}$, $b_i^{(k_0)}(\z)$, and $e_h^{(k_0)}(\z)$, respectively, we may assume that $\Omega^k\ba\in U$ for all $k\geq0$ and that the multiplicative subgroup $G$ of $K^\times$ generated by the eigenvalues of $A$ is torsion free. Since $e_h(\Omega^k\ba)\neq0$ $(1\leq h\leq m)$ for all $k\geq0$, by the functional equation \eqref{eq:g} and the condition (IV)$_p$, we see that $g_h(\Omega^k\ba)\neq0$  $(1\leq h\leq m)$ for all $k\geq0$.

To prove the theorem, we assume on the contrary that $f_i(\ba)$ $(1\leq i\leq l)$ and $g_h(\ba)$ $(1\leq h\leq m)$ are algebraically dependent over $\Q$. Then there exist a positive integer $L$ and integers $\tau_{\bl\m}$ $(\bl\in\La,\ \m\in\M)$, not all zero, such that
$$\sum_{\substack{\bl\in\La\\ \m\in\M}}\tau_{\bl\m}\f(\ba)^{\bl}\g(\ba)^{\m}=0,$$
where $\La:=\{\bl\in\N^l\mid |\bl|\leq L\}$ and $\M:=\{0,1,\ldots,L\}^m$. Let $x_{ij}$ $(1\leq i,j\leq l)$, $w_i$ $(1\leq i\leq l)$, $y_i$ $(1\leq i\leq l)$, $x'_h$ $(1\leq h\leq m)$, $w'_h$ $(1\leq h\leq m)$, and $t_{\bl\m}$ $(\bl\in\La,\ \m\in\M)$ be variables and let
\begin{gather*}
X:=\left(
\begin{array}{ccc}
x_{11}&\cdots&x_{1l} \\
\vdots&&\vdots \\
x_{l1}&\cdots&x_{ll}
\end{array}
\right),\quad
\bm{w}:=\left(
\begin{array}{c}
w_1 \\
\vdots \\
w_l
\end{array}
\right),\quad
\bm{y}:=\left(
\begin{array}{c}
y_1 \\
\vdots \\
y_l
\end{array}
\right),\\
\bm{x}':=\left(
\begin{array}{c}
x'_1 \\
\vdots \\
x'_m
\end{array}
\right),\quad
\bm{w}':=\left(
\begin{array}{c}
w'_1 \\
\vdots \\
w'_m
\end{array}
\right),\quad
\bm{x}'\bm{w}':=\left(
\begin{array}{c}
x'_1w'_1 \\
\vdots \\
x'_mw'_m
\end{array}
\right),
\end{gather*}
and
$$F(\z;\bt):=\sum_{\substack{\bl\in\La\\ \m\in\M}}t_{\bl\m}\f(\z)^{\bl}\g(\z)^{\m}.$$
We define $T_{\bl\m}(\bt;X;\bm{y};\bm{x}')$ $(\bl\in\La,\ \m\in\M)$ by the equality
$$\sum_{\substack{\bl\in\La\\ \m\in\M}}t_{\bl\m}(X\bm{w}+\bm{y})^{\bl}(\bm{x}'\bm{w}')^{\m}=:\sum_{\substack{\bl\in\La\\ \m\in\M}} T_{\bl\m}(\bt;X;\bm{y};\bm{x}')\bm{w}^{\bl}\bm{w}'^{\m},$$
namely,
\begin{align*}
&T_{\bl\m}(\bt;X;\bm{y};\bm{x}')\\
&=\bm{x}'^{\m}\sum_{\substack{\bm{\nu}=(\nu_1,\ldots,\nu_l)\in\N^l\\ |\bl|\leq|\bm{\nu}|\leq L}}t_{\bm{\nu}\m}
\sum_{\substack{\bm{\nu}_1,\ldots,\bm{\nu}_l\in\N^{l+1}\\ \bm{\nu}_i=(\nu_{i0},\nu_{i1},\ldots,\nu_{il})\\ |\bm{\nu}_i|=\nu_i\ (1\leq i\leq l)\\ \sum_{i=1}^l\nu_{ij}=\lambda_j\ (1\leq j\leq l)}}\prod_{i=1}^l\binom{\nu_i}{\nu_{i0}\ \nu_{i1}\,\cdots\,\nu_{il}}y_i^{\nu_{i0}}x_{i1}^{\nu_{i1}}\cdots x_{il}^{\nu_{il}}
\end{align*}
for any $\bl=(\lambda_1,\ldots,\lambda_l)\in\La$ and $\m\in\M$. Letting
$$\e^{(k)}(\z):={}^t(e_1^{(k)}(\z),\ldots,e_m^{(k)}(\z))$$
and
$$\e^{(k)}(\z)\g(\Omega^k\z):={}^t(e_1^{(k)}(\z)g_1(\Omega^k\z),\ldots,e_m^{(k)}(\z)g_m(\Omega^k\z)),$$
by the functional equations \eqref{eq:fk} and \eqref{eq:gk}, we have
\begin{align*}
F(\z;\bt)&=\sum_{\substack{\bl\in\La\\ \m\in\M}}t_{\bl\m}\f(\z)^{\bl}\g(\z)^{\m}\\
&=\sum_{\substack{\bl\in\La\\ \m\in\M}}t_{\bl\m}(A^k\f(\Omega^k\z)+\bb^{(k)}(\z))^{\bl}(\e^{(k)}(\z)\g(\Omega^k\z))^{\m}\\
&=\sum_{\substack{\bl\in\La\\ \m\in\M}} T_{\bl\m}(\bt;A^k;\bb^{(k)}(\z);\e^{(k)}(\z))\f(\Omega^k\z)^{\bl}\g(\Omega^k\z)^{\m}\\
&=F(\Omega^k\z;\bm{T}(\bt;A^k;\bb^{(k)}(\z);\e^{(k)}(\z)))
\end{align*}
for all $k\geq0$. Hence
\begin{equation}\label{eq:vanish}
F(\Omega^k\ba;\bm{T}(\bm{\tau};A^k;\bb^{(k)}(\ba);\e^{(k)}(\ba)))=F(\ba;\bm{\tau})=0\quad(k\geq0).
\end{equation}
We define an ideal $V(\bm{\tau})$ of $K[\bt]$ by
$$V(\bm{\tau}):=\{Q(\bt)\in K[\bt]\mid Q(\bm{T}(\bm{\tau};A^k;\bm{y};\bm{x}'))=0\ {\rm for\ all}\ k\geq0\}.$$

\begin{lem}\label{lem:prime}
$V(\bm{\tau})$ is a prime ideal of $K[\bt]$.
\end{lem}

For the proof we use the following

\begin{lem}[Skolem-Lech-Mahler, cf. Nishioka \cite{N}]\label{lem:slm}
Let $C$ be a field of characteristic $0$. Let $\gamma_1,\dots,\gamma_s$ be nonzero distinct elements of $C$ and $P_1(X),\ldots,P_s(X)\in C[X]$ nonzero polynomials. Then, if $\{k\in\N\mid \sum_{i=1}^sP_i(k)\gamma_i^k=0\}$ is an infinite set, then $\gamma_i/\gamma_j$ is a root of unity for some distinct $i,j$.
\end{lem}

\begin{proof}[Proof of Lemma \ref{lem:prime}]
We define a subset $\mathcal{R}_1$ of $(K[\bm{y};\bm{x}'])^\N$ by
$$\mathcal{R}_1:=\Set{\left\{\sum_{\gamma\in\Gamma}p_{\gamma}(k)\gamma^k\right\}_{k\geq0}\ | \begin{array}{l}\Gamma {\rm\ is\ a\ finite\ subset\ of\ }G{\rm\ independent\ of\ }k,\\p_\gamma(Y)\in(K[\bm{y};\bm{x}'])[Y]\ (\gamma\in\Gamma)\end{array}}.$$
Then $\mathcal{R}_1$ forms a commutative ring including $K[\bm{y};\bm{x}']$ under termwise addition and multiplication. If we put $A^k=:(a_{ij}^{(k)})$, then $\{a_{ij}^{(k)}\}_{k\geq0}\in \mathcal{R}_1$ for any $1\leq i,j\leq l$. Since $T_{\bl\m}(\bm{\tau};X;\bm{y};\bm{x}')\in(\Z[\bm{y};\bm{x}'])[\{x_{ij}\}]$, we have $\{T_{\bl\m}(\bm{\tau};A^k;\bm{y};\bm{x}')\}_{k\geq0}\in \mathcal{R}_1$ for any $\bl\in\La$ and $\m\in\M$. Therefore, if $P(\bt)\in K[\bt]$, then $\{P(\bm{T}(\bm{\tau};A^k;\bm{y};\bm{x}'))\}_{k\geq0}\in \mathcal{R}_1$, so that there exist a finite subset $\Gamma=\Gamma(P)$ of $G$ and nonzero polynomials $p_{\gamma}(Y)\in(K[\bm{y};\bm{x}'])[Y]$ $(\gamma\in\Gamma)$ such that
$$P(\bm{T}(\bm{\tau};A^k;\bm{y};\bm{x}'))=\sum_{\gamma\in\Gamma}p_\gamma(k)\gamma^k$$
for all $k\geq0$.

To prove the lemma, we let $P_1(\bt),P_2(\bt)\in K[\bt]$ and suppose that $P_1(\bt)P_2(\bt)\in V(\bm{\tau})$. Since $P_1(\bm{T}(\bm{\tau};A^k;\bm{y};\bm{x}'))P_2(\bm{T}(\bm{\tau};A^k;\bm{y};\bm{x}'))=0$ for all $k\geq0$, we may assume that $P_1(\bm{T}(\bm{\tau};A^k;\bm{y};\bm{x}'))=0$ for infinitely many $k$. Hence, if $\Gamma(P_1)\neq\emptyset$, then Lemma \ref{lem:slm} implies that there exist distinct $\gamma,\gamma'\in\Gamma(P_1)$ such that $\gamma/\gamma'$ is a root of unity, which contradicts the fact that $G$ is torsion free. Thus $\Gamma(P_1)=\emptyset$ and $P_1(\bt)\in V(\bm{\tau})$.
\end{proof}

\begin{prop}\label{prop:nonzero}
The following two conditions are equivalent for any $P(\z;\bt)\in K[\z;\bt]$.
\begin{enumerate}
\item $P(\Omega^k\ba;\bm{T}(\bm{\tau};A^k;\bb^{(k)}(\ba);\e^{(k)}(\ba)))=0$ for all sufficiently large $k$.
\item If we put $P(\z;\bt)=:\sum_{\bm{\eta}\in\Eta}Q_{\bm{\eta}}(\bt)\z^{\bm{\eta}}$, where $Q_{\bm{\eta}}(\bt)\in K[\bt]$ $(\eta\in\Eta)$ and $\Eta$ is a finite subset of $\N^n$, then $Q_{\bm{\eta}}(\bt)\in V(\bm{\tau})$ for any $\bm{\eta}\in\Eta$.
\end{enumerate}
\end{prop}

\begin{proof}
We only prove that the condition (i) implies (ii) since the converse is trivial. We define a subset $\mathcal{R}_2$ of $(\Qbar_p[w_1,\ldots,w_l,\frac{1}{w'_1},\ldots,\frac{1}{w'_m}])^\N$ by
\begin{gather*}
\mathcal{R}_2:=\Set{\left\{\sum_{\gamma\in\Gamma}q_{\gamma}(k)\gamma^k\right\}_{k\geq0}\ | \begin{array}{l}\Gamma {\rm\ is\ a\ finite\ subset\ of\ }G{\rm\ independent\ of\ }k,\\q_\gamma(Y)\in (\Qbar_p[w_1,\ldots,w_l,\frac{1}{w'_1},\ldots,\frac{1}{w'_m}])[Y]\ (\gamma\in\Gamma)\end{array}}\\
=\Set{\left\{\sum_{\substack{\bm{\nu}\in\Nu\\\bm{\xi}\in\X}}\left(\sum_{\gamma\in\Gamma}r_{\bm{\nu}\bm{\xi}\gamma}(k)\gamma^k\right)\bm{w}^{\bm{\nu}}\bm{w}'^{-\bm{\xi}}\right\}_{k\geq0}\ |\begin{array}{l}\Nu\subset\mathbb{N}^l,\ \X\subset\mathbb{N}^m,\ \Gamma\subset G\\{\rm are\ finite\ sets \ independent\ of\ } k,\\r_{\bm{\nu}\bm{\xi}\gamma}(Y)\in\Qbar_p[Y]\\(\bm{\nu}\in\Nu,\ \bm{\xi}\in\X,\ \gamma\in\Gamma)\end{array}}.
\end{gather*}
Then $\mathcal{R}_2$ forms a commutative ring including $\Qbar_p[w_1,\ldots,w_l,\frac{1}{w'_1},\ldots,\frac{1}{w'_m}]$ under term-wise addition and multiplication. In the same way as in the proof of Lemma \ref{lem:prime}, we see that $\{Q_{\bm{\eta}}(\bm{T}(\bm{\tau};A^k;\f(\ba)-A^k\bm{w};\g(\ba)/\bm{w}'))\}_{k\geq0}\in \mathcal{R}_2$ for any $\bm{\eta}\in\Eta$, where $\g(\ba)/\bm{w}':={}^t(g_1(\ba)/w'_1,\ldots,g_m(\ba)/w'_m)$. Hence there exist finite sets $\Nu\subset\N^l$ and $\X\subset\N^m$, distinct elements $\gamma_1,\ldots,\gamma_q$ of $G$, and positive integers $d_1,\ldots,d_q$ such that
$$
Q_{\bm{\eta}}(\bm{T}(\bm{\tau};A^k;\f(\ba)-A^k\bm{w};\g(\ba)/\bm{w}'))=\sum_{\substack{\bm{\nu}\in\Nu\\\bm{\xi}\in\X}}R_{\bm{\eta}\bm{\nu}\bm{\xi}}(k)\bm{w}^{\bm{\nu}}\bm{w}'^{-\bm{\xi}}
$$
for all $k\geq0$ and $\bm{\eta}\in\Eta$, where
$$
R_{\bm{\eta}\bm{\nu}\bm{\xi}}(k)=\sum_{i=1}^q\sum_{j=1}^{d_i}r_{\bm{\eta}\bm{\nu}\bm{\xi}ij}k^{j-1}\gamma_i^{k},\quad r_{\bm{\eta}\bm{\nu}\bm{\xi}ij}\in\Qbar_p.
$$

We claim that every $\{R_{\bm{\eta}\bm{\nu}\bm{\xi}}(k)\}_{k\geq0}$ is the null sequence. Since $g_h(\bm{0})\neq0$ $(1\leq h\leq m)$, 
$$h_{ij}(\bm{z}):=\sum_{\bm{\eta}\in\Eta}\sum_{\substack{\bm{\nu}\in\Nu\\\bm{\xi}\in\X}}r_{\bm{\eta}\bm{\nu}\bm{\xi}ij}\f(\z)^{\bm{\nu}}\g(\z)^{-\bm{\xi}}\z^{\bm{\eta}}\quad(1\leq i\leq q,\ 1\leq j\leq d_i)$$
are formal power series in the variables $z_1,\ldots,z_n$ with coefficients in $\Qbar_p$ which converge in an $n$-polydisc around the origin of $\C_p^n$. Define
$$\psi(\z;x):=\sum_{i=1}^q\sum_{j=1}^{d_i}x^{j-1}\gamma_i^{x}h_{ij}(\z).$$
By the condition (i) of the proposition and the functional equations \eqref{eq:fk} and \eqref{eq:gk}, we see that
\begin{align*}
0&=P(\Omega^k\ba;\bm{T}(\bm{\tau};A^k;\bb^{(k)}(\ba);\e^{(k)}(\ba)))\\
&=\sum_{\bm{\eta}\in\Eta}Q_{\bm{\eta}}(\bm{T}(\bm{\tau};A^k;\bb^{(k)}(\ba);\e^{(k)}(\ba)))(\Omega^k\ba)^{\bm{\eta}}\\
&=\sum_{\bm{\eta}\in\Eta}\left(\sum_{\substack{\bm{\nu}\in\Nu\\\bm{\xi}\in\X}}R_{\bm{\eta}\bm{\nu}\bm{\xi}}(k)\f(\Omega^k\ba)^{\bm{\nu}}\g(\Omega^k\ba)^{-\bm{\xi}}\right)(\Omega^k\ba)^{\bm{\eta}}\\
&=\psi(\Omega^k\ba;k)
\end{align*}
for all sufficiently large $k$. Then Lemma \ref{lem:crucial} implies that $h_{ij}(\z)=0$ for any $1\leq i\leq q$ and $1\leq j\leq d_i$. Therefore, noting that $f_1(\z),\ldots,f_l(\z),g_1(\z),\ldots,g_m(\z)$ are algebraically independent over $\Qbar_p(z_1,\ldots,z_n)$ (cf. Nishioka \cite[p. 6]{N}), we have $r_{\bm{\eta}\bm{\nu}\bm{\xi}ij}=0$ for any $\bm{\eta}$, $\bm{\nu}$, $\bm{\xi}$, $i$, and $j$. This proves our claim.

By the claim we have
$$Q_{\bm{\eta}}(\bm{T}(\bm{\tau};A^k;\f(\ba)-A^k\bm{w};\g(\ba)/\bm{w}'))=0$$
for all $k\geq0$ and $\bm{\eta}\in\Eta$. Noting that $\det A\neq0$ and that $g_h(\ba)\neq0$ $(1\leq h\leq m)$, we obtain
$$Q_{\bm{\eta}}(\bm{T}(\bm{\tau};A^k;\bm{y};\bm{x}'))=0$$
for all $k\geq0$ and $\bm{\eta}\in\Eta$, which implies the condition (ii) of the proposition.
\end{proof} 

\begin{defn}{\rm
For $P(\z;\bt)=\sum_{\bm{\eta}\in\N^n}P_{\bm{\eta}}(\bt)\z^{\bm{\eta}}\in (K[\bt])[[z_1,\ldots,z_n]]$ we define
$$\ind P(\z;\bt):=\min\{|\bm{\eta}|\mid P_{\bm{\eta}}(\bt)\notin V(\bm{\tau})\},$$
where $\min\emptyset:=\infty$.}
\end{defn}

In what follows, $c_1,c_2,\ldots$ denote positive constants independent of $N$ and $k$. If they depend on $N$, then we denote them by $c_1(N),c_2(N),\ldots$. We denote by $\mathcal{O}_K$ the ring of algebraic integers of $K$. The following proposition is proved in the same way as in the proof of Proposition 5 in Nishioka \cite{N1996} by using the condition (IV)$_p$, the fact that $F(\z;\bm{\tau})\not\equiv0$, and Lemma \ref{lem:prime}.

\begin{prop}\label{prop:Auxiliary}
Let $N$ be a sufficiently large positive integer. Then there exist $N+1$ polynomials $P_0(\z;\bt),\ldots,P_N(\z;\bt)\in \mathcal{O}_K[\z;\bt]$ with degree at most $N$ in each of the variables $z_i,t_{\bl\m}$ $(1\leq i\leq n,\ \bl\in\La,\ \m\in\M)$ such that the following two conditions are satisfied.
\begin{enumerate}
\item $\ind P_0(\z;\bt)<\infty.$
\item $\ind(\sum_{h=0}^NP_h(\z;\bt)F(\z;\bt)^h)\geq c_1(N+1)^{1+1/n}.$
\end{enumerate}
\end{prop}

Let $E(\z;\bt)$ be the $\sum_{h=0}^NP_h(\z;\bt)F(\z;\bt)^h$ in Proposition \ref{prop:Auxiliary} and $\rho$ the maximum of the archimedean absolute values of the eigenvalues of $\Omega$.

\begin{prop}\label{prop:UB}
If $k>c_2(N)$, then
$$\log|E(\Omega^k\ba;\bm{T}(\bm{\tau};A^k;\bb^{(k)}(\ba);\e^{(k)}(\ba)))|_p\leq -c_3(N+1)^{1+1/n}\rho^k.$$
\end{prop}

\begin{proof}
Since $f_j(\Omega^k\ba)\to f_j(\bm{0})$ $(k\to\infty)$ for $1\leq j\leq l$, by the functional equation \eqref{eq:fk} we have $|b_i^{(k)}(\ba)|_p\leq c_4^k$ for $1\leq i\leq l$. Similarly, since $g_h(\Omega^k\ba)\to g_h(\bm{0})\neq0$ $(k\to\infty)$ for $1\leq h\leq m$, by the functional equation \eqref{eq:gk} we have $|e_h^{(k)}(\ba)|_p\leq c_5$ for $1\leq h\leq m$. Hence $|T_{\bl\m}(\bm{\tau};A^k;\bb^{(k)}(\ba);\e^{(k)}(\ba))|_p\leq c_6^k$ for $\bl\in\La$ and $\m\in\M$. We note that $E(\z;\bt)$ is a polynomial in the variables $t_{\bl\m}$ $(\bl\in\La,\ \m\in\M)$ with degree at most $2N$ in each variable whose coefficients are power series convergent in $U$. Let
$$E(\z;\bt)=:\sum_{\bm{\nu}\in\{0,1,\ldots,2N\}^s}h_{\bm{\nu}}(\z)\bt^{\bm{\nu}},\quad h_{\bm{\nu}}(\z)=:\sum_{\bm{\xi}\in\N^n}h_{\bm{\nu}\bm{\xi}}\z^{\bm{\xi}}\in K[[\z]],$$
where $s:=\sharp\La\times\sharp\M=\binom{L+l}{l}(L+1)^m$. Then we have
$$|h_{\bm{\nu}\bm{\xi}}|_p\leq c_7(N)c_8^{|\bm{\xi}|}\quad (\bm{\nu}\in\{0,1,\ldots,2N\}^s,\ \bm{\xi}\in\N^n)$$
and
$$E(\z;\bt)=\sum_{\bm{\xi}\in\N^n}\left(\sum_{\bm{\nu}\in\{0,1,\ldots,2N\}^s}h_{\bm{\nu}\bm{\xi}}\bt^{\bm{\nu}}\right)\z^{\bm{\xi}}.$$
Therefore
$$
|E(\Omega^k\ba;\bm{T}(\bm{\tau};A^k;\bb^{(k)}(\ba);\e^{(k)}(\ba)))|_p\leq c_9(N)c_{10}^{Nk}\sum_{\substack{\bm{\xi}\in\N^n\\|\bm{\xi}|\geq I}}c_8^{|\bm{\xi}|}|(\Omega^k\ba)^{\bm{\xi}}|_p,
$$
where $I:=\ind E(\z;\bt)$. By the condition (III)$_p$, there exists a positive constant $\theta<1$ such that $|\alpha_i^{(k)}|_p\leq\theta^{\rho^k}$ for $1\leq i\leq n$ and for all sufficiently large $k$. Hence
\begin{align*}
|E(\Omega^k\ba;\bm{T}(\bm{\tau};A^k;\bb^{(k)}(\ba);\e^{(k)}(\ba)))|_p&\leq c_9(N)c_{10}^{Nk}\sum_{i=1}^n\ \sum_{\substack{\bm{\xi}=(\xi_1,\ldots,\xi_n)\in\N^n\\\xi_i\geq I/n}}(c_8\theta^{\rho^k})^{|\bm{\xi}|}\\
&\leq nc_9(N)c_{10}^{Nk}(c_8\theta^{\rho^k})^{I/n}/(1-c_8\theta^{\rho^k})^n.
\end{align*}
Since $I\geq c_1(N+1)^{1+1/n}$ by the condition (ii) of Proposition \ref{prop:Auxiliary}, we see that, if $k\geq c_2(N)$, then
$$\log|E(\Omega^k\ba;\bm{T}(\bm{\tau};A^k;\bb^{(k)}(\ba);\e^{(k)}(\ba)))|_p\leq -c_3(N+1)^{1+1/n}\rho^k.$$
\end{proof}

\begin{prop}\label{prop:LB}
If $k>c_4(N)$, then
$$\log\|E(\Omega^k\ba;\bm{T}(\bm{\tau};A^k;\bb^{(k)}(\ba);\e^{(k)}(\ba)))\|\leq c_5N\rho^k.$$
\end{prop}

\begin{proof}
From \eqref{eq:vanish} we have
$$
E(\Omega^k\ba;\bm{T}(\bm{\tau};A^k;\bb^{(k)}(\ba);\e^{(k)}(\ba)))=P_0(\Omega^k\ba;\bm{T}(\bm{\tau};A^k;\bb^{(k)}(\ba);\e^{(k)}(\ba))).
$$
Letting $A^k=:(a_{ij}^{(k)})$, we have $\|a_{ij}^{(k)}\|\leq c_6^k$ for $1\leq i,j\leq l$. By the condition (II) we see that $\|b_i(\Omega^k\ba)\|\leq c_{7}^{\rho^k}$ for $1\leq i\leq l$ and that $\|e_h(\Omega^k\ba)\|\leq c_{8}^{\rho^k}$ for $1\leq h\leq m$. Hence we have
$$
\|b_i^{(k)}(\ba)\|\leq kl\prod_{j=0}^{k-1}(c_6^jc_{7}^{\rho^j})^l\leq c_{9}^{\rho^k}\quad(1\leq i\leq l)
$$
and
$$
\|e_h^{(k)}(\ba)\|\leq\prod_{j=0}^{k-1}c_{8}^{\rho^j}\leq c_{10}^{\rho^k}\quad(1\leq h\leq m)
$$
by \eqref{eq:bk} and \eqref{eq:ek}, respectively. Therefore
$$
\|T_{\bl\m}(\bm{\tau};A^k;\bb^{(k)}(\ba);\e^{(k)}(\ba))\|\leq c_{11}^{\rho^k}
$$
for $\bl\in\La$ and $\m\in\M$. Since the degree of each variable of $P_0(\z;\bt)\in \mathcal{O}_K[\z;\bt]$ is at most $N$, we obtain
$$\|P_0(\Omega^k\ba;\bm{T}(\bm{\tau};A^k;\bb^{(k)}(\ba);\e^{(k)}(\ba)))\|\leq c_{12}(N)c_{13}^{N\rho^k}.$$
This implies the proposition.
\end{proof}

\noindent {\it Completion of the proof of Theorem \ref{thm:indep_values}.}
By the condition (i) of Proposition \ref{prop:Auxiliary} together with Proposition \ref{prop:nonzero}, there exists a positive integer $k$ greater than both $c_2(N)$ and $c_4(N)$ such that
$$
E(\Omega^k\ba;\bm{T}(\bm{\tau};A^k;\bb^{(k)}(\ba);\e^{(k)}(\ba)))=P_0(\Omega^k\ba;\bm{T}(\bm{\tau};A^k;\bb^{(k)}(\ba);\e^{(k)}(\ba)))\neq0.
$$
Therefore, by Propositions \ref{prop:UB}, \ref{prop:LB}, and the fundamental inequality \eqref{eq:Liouville}, we have
$$
-c_3(N+1)^{1+1/n}\rho^k\geq-2[K:\Q]c_5N\rho^k.$$
Hence
$$c_3(N+1)^{1+1/n}\leq 2[K:\Q]c_5N,$$
which is a contradiction if $N$ is large.
\end{proof}

\section{Proof of Theorem \ref{thm:main2}}\label{sec:4}
In this section, we prove Theorem \ref{thm:main2} by using Theorem \ref{thm:criterion} and Tanaka's results. Let $\{R_k\}_{k\geq0}$ be a linear recurrence of nonnegative integers satisfying \eqref{eq:LRS} and let
\begin{equation}\label{eq:Omega}
\Omega:=\left(
\begin{array}{ccccc}
c_1&1&0&\cdots&0\\
c_2&0&1&\ddots&\vdots\\
\vdots&\vdots&\ddots&\ddots&0\\
\vdots&\vdots&&\ddots&1\\
c_n&0&\cdots&\cdots&0
\end{array}
\right).
\end{equation}
We define a monomial
\begin{equation}\label{eq:monomial}
M(\z):=z_1^{R_{n-1}}\cdots z_n^{R_0},
\end{equation}
which is denoted similarly to \eqref{eq:Omegaz} by
\begin{equation}\label{eq:monomial2}
M(\z)=(R_{n-1},\ldots,R_0)\z.
\end{equation}
It follows from \eqref{eq:LRS}, \eqref{eq:Omegaz}, and \eqref{eq:monomial2} that
$$
M(\Omega^k\z)=z_1^{R_{k+n-1}}\cdots z_n^{R_k}\quad(k\geq0).
$$

\begin{lem}\label{lem:conditions}
Let $p$ be $\infty$ or a prime number. Suppose that $\{R_k\}_{k\geq0}$ satisfies the condition {\rm (R)}$_p$ stated in Section \ref{sec:1}. Then, if $\alpha$ is an algebraic number with $0<|\alpha|_p<1$, then the matrix $\Omega$ defined by \eqref{eq:Omega} and the point $\ba=(\underbrace{1,\ldots,1}_{n-1},\alpha)$ satisfy the conditions {\rm(I)}, {\rm(II)}, {\rm(III)}$_p$, and {\rm(IV)}$_p$ stated in Section \ref{sec:3.1}.
\end{lem}

In the case where $p$ is $\infty$, Lemma \ref{lem:conditions} was proved by Tanaka \cite{T1996}, and in the case where $p$ is a prime number, it can be proved by using Lemma \ref{lem:Mahler} (cf. Mahler \cite{Mahler}). The following lemmas are central to the proof of Theorem \ref{thm:main2}.

\begin{lem}[A special case of Theorem 1 of Tanaka \cite{T1999}]\label{lem:R}
Let $\Cbar$ be an algebraically closed field of characteristic $0$. Suppose that $\{R_k\}_{k\geq0}$ satisfies the condition {\rm (R)}$_\infty$ stated in Section \ref{sec:1}. Assume that $R(\z)\in\Cbar[[z_1,\ldots,z_n]]$ satisfies the functional equation of the form
$$R(\z)=\alpha R(\Omega\z)+Q(M(\z)),$$
where $\alpha\neq0$ is an element of $\Cbar$, $\Omega$ is defined by \eqref{eq:Omega}, $M(\z)$ is defined by \eqref{eq:monomial}, and $Q(X)\in\Cbar(X)$ is defined at $X=0$. Then, if $R(\z)\in\Cbar(z_1,\ldots,z_n)$, then $R(\z)\in\Cbar$ and $Q(X)\in\Cbar$.
\end{lem}

\begin{lem}[A special case of Theorem 2 of Tanaka \cite{T1999}]\label{lem:S}
Let $\Cbar$ be an algebraically closed field of characteristic $0$. Suppose that $\{R_k\}_{k\geq0}$ satisfies the condition {\rm (R)}$_\infty$ stated in Section \ref{sec:1}. Assume that $S(\z)$ is a nonzero element of the quotient field of $\Cbar[[z_1,\ldots,z_n]]$ satisfying the functional equation of the form
$$S(\z)=Q(M(\z))S(\Omega\z),$$
where $\Omega$ is defined by \eqref{eq:Omega}, $M(\z)$ is defined by \eqref{eq:monomial}, and $Q(X)\in\Cbar(X)$ is defined and nonzero at $X=0$. Then, if $S(\z)\in\Cbar(z_1,\ldots,z_n)^\times$, then $S(\z)\in\Cbar^\times$ and $Q(X)=1$.
\end{lem}

\begin{proof}[Proof of Theorem \ref{thm:main2}]
Assume on the contrary that there exist distinct $\alpha_1,\ldots,\alpha_r\in\Qbar^\times$, distinct $\beta_1,\ldots,\beta_s\in\Qbar^\times\setminus\{a^{-R_k}\}_{k\geq0}$, and nonnegative integers $L,M$ such that the values
\begin{align}\label{eq:values}
&\frac{\partial^{l+m}H}{\partial x^l\partial y^m}(\alpha_i,\beta_j)\quad(1\leq i\leq r,\ 1\leq j\leq s,\ 0\leq l\leq L,\ 0\leq m\leq M),\nonumber\\
F_m^{(l)}(\alpha_i)&\quad(1\leq i\leq r,\ 0\leq l\leq L,\ 0\leq m\leq M),\quad {\rm and} \quad\, G(\beta_j)\quad(1\leq j\leq s)
\end{align}
are algebraically dependent over $\Q$. Shifting the linear recurrence $\{R_k\}_{k\geq0}$, we may suppose that $R_k$ $(k\geq0)$ are sufficiently large. Let $z_1,\ldots,z_n$ be variables and let $\z:=(z_1,\ldots,z_n)$. Define
\begin{align*}
h_{jm}(x;\z):=\sum_{k=0}^\infty x^k\left(\frac{M(\Omega^k\z)}{1-\beta_jM(\Omega^k\z)}\right)^{m+1}\quad &(1\leq j\leq s,\ 0\leq m\leq M),\\
f_m(x;\z):=\sum_{k=0}^\infty x^kM(\Omega^k\z)^{m+1}\quad&(0\leq m\leq M),
\end{align*}
and
$$g_j(\z):=\prod_{k=0}^\infty \left(1-\beta_jM(\Omega^k\z)\right)\quad(1\leq j\leq s),$$
where $\Omega$ and $M(\z)$ are defined by \eqref{eq:Omega} and \eqref{eq:monomial}, respectively. Furthermore, define
$$h_{ijlm}(\z):=\frac{\partial^lh_{jm}}{\partial x^l}(\alpha_i;\z)\quad(1\leq i\leq r,\ 1\leq j\leq s,\ 0\leq l\leq L,\ 0\leq m\leq M)$$
and
$$f_{ilm}(\z):=\frac{\partial^lf_m}{\partial x^l}(\alpha_i;\z)\quad(1\leq i\leq r,\ 0\leq l\leq L,\ 0\leq m\leq M).$$
Since $M(\Omega^k\z)=z_1^{R_{k+n-1}}\cdots z_n^{R_k}$, letting
\begin{equation}\label{eq:gamma}
\bg:=(\underbrace{1,\ldots,1}_{n-1},a),
\end{equation}
we see that
\begin{align*}
h_{ijlm}(\bg)=\frac{1}{m!}\frac{\partial^{l+m}H}{\partial x^l\partial y^m}(\alpha_i,\beta_j)\quad&(1\leq i\leq r,\ 1\leq j\leq s,\ 0\leq l\leq L,\ 0\leq m\leq M),\\
f_{ilm}(\bg)=\frac{1}{m!}F_m^{(l)}(\alpha_i)\quad&(1\leq i\leq r,\ 0\leq l\leq L,\ 0\leq m\leq M),
\end{align*}
and
$$g_j(\bg)=G(\beta_j)\quad(1\leq j\leq s).$$
Since the values \eqref{eq:values} are algebraically dependent over $\Q$, so are the values
\begin{gather*}
h_{ijlm}(\bg)\quad(1\leq i\leq r,\ 1\leq j\leq s,\ 0\leq l\leq L,\ 0\leq m\leq M),\\
f_{ilm}(\bg)\quad(1\leq i\leq r,\ 0\leq l\leq L,\ 0\leq m\leq M),\qquad {\rm and} \qquad g_j(\bg)\quad(1\leq j\leq s).
\end{gather*}
Here we see that
$$h_{jm}(x;\z)=xh_{jm}(x;\Omega\z)+\left(\frac{M(\z)}{1-\beta_jM(\z)}\right)^{m+1}\quad(1\leq j\leq s,\ 0\leq m\leq M),$$
and hence
\begin{align*}
\frac{\partial^l h_{jm}}{\partial x^l}(x;\z)=x\frac{\partial^l h_{jm}}{\partial x^l}(x;\Omega\z)+l\frac{\partial^{l-1} h_{jm}}{\partial x^{l-1}}(x;\Omega\z)\\
(1\leq j\leq s,\ 1\leq l\leq L,\ 0\leq m\leq M).
\end{align*}
Thus we see that $h_{ijlm}(\z)$ $(1\leq i\leq r,\ 1\leq j\leq s,\ 0\leq l\leq L,\ 0\leq m\leq M)$ satisfy the functional equations
\begin{align*}
\left(
\begin{array}{c}
h_{ij0m}(\z) \\
h_{ij1m}(\z) \\
\vdots \\
h_{ijLm}(\z)
\end{array}
\right)=A_i\left(
\begin{array}{c}
h_{ij0m}(\Omega\z) \\
h_{ij1m}(\Omega\z) \\
\vdots \\
h_{ijLm}(\Omega\z)
\end{array}
\right)+\left(
\begin{array}{c}
(M(\z)/(1-\beta_jM(\z)))^{m+1}\\
0 \\
\vdots \\
0
\end{array}
\right)\\
(1\leq i\leq r,\ 1\leq j\leq s,\ 0\leq m\leq M),
\end{align*}
where
$$
A_i:=\left(
\begin{array}{ccccc}
\alpha_i&&&&\\
1&\alpha_i&&&\hsymb{0}\\
&2&\ddots&&\\
&&\ddots&\ddots&\\
\hsymb{0}&&&L&\alpha_i
\end{array}
\right)\quad(1\leq i\leq r).
$$
Similarly, we see that $f_{ilm}(\z)$ $(1\leq i\leq r,\ 0\leq l\leq L,\ 0\leq m\leq M)$ satisfy the functional equations
$$
\left(
\begin{array}{c}
f_{i0m}(\z) \\
f_{i1m}(\z) \\
\vdots \\
f_{iLm}(\z)
\end{array}
\right)=A_i\left(
\begin{array}{c}
f_{i0m}(\Omega\z) \\
f_{i1m}(\Omega\z) \\
\vdots \\
f_{iLm}(\Omega\z)
\end{array}
\right)+\left(
\begin{array}{c}
M(\z)^{m+1} \\
0 \\
\vdots \\
0
\end{array}
\right)\quad(1\leq i\leq r,\ 0\leq m\leq M).
$$
Furthermore, we see that $g_j(\z)$ $(1\leq j\leq s)$ satisfy the functional equations
$$g_j(\z)=(1-\beta_jM(\z))g_j(\Omega\z)\quad(1\leq j\leq s).$$
By Lemma \ref{lem:conditions}, the matrix $\Omega$ and the point $\bg$ defined by \eqref{eq:gamma} satisfy the conditions {\rm(I)}, {\rm(II)}, {\rm(III)}$_p$, and {\rm(IV)}$_p$ stated in Section \ref{sec:3.1}. Therefore, since $\alpha_i$ $(1\leq i\leq r)$ are distinct, by Theorem \ref{thm:criterion} at least one the following two cases arises:
\begin{enumerate}
\item There exist $i\in\{1,\ldots,r\}$, $c_{jm}\in\Qbar$ $(0\leq j\leq s,\ 0\leq m\leq M)$, not all zero, and $R(\z)\in \Qbar[[z_1,\ldots,z_n]]\cap\Qbar(z_1,\ldots,z_n)$ such that
\begin{equation}\label{eq:RR}
R(\z)=\alpha_iR(\Omega\z)+\sum_{m=0}^Mc_{0m}M(\z)^{m+1}+\sum_{j=1}^s\sum_{m=0}^M c_{jm}\left(\frac{M(\z)}{1-\beta_jM(\z)}\right)^{m+1}.
\end{equation}
\item There exist integers $d_1,\ldots,d_s$, not all zero, and $S(\z)\in\Qbar(z_1,\ldots,z_n)^\times$ such that
\begin{equation}\label{eq:SS}
S(\z)=S(\Omega\z)\prod_{j=1}^s(1-\beta_jM(\z))^{d_j}.
\end{equation}
\end{enumerate}
If the functional equation \eqref{eq:RR} is satisfied, then by Lemma \ref{lem:R}
\begin{equation}\label{eq:RRR}
\sum_{m=0}^Mc_{0m}X^{m+1}+\sum_{j=1}^s\sum_{m=0}^M c_{jm}\left(\frac{X}{1-\beta_jX}\right)^{m+1}=\delta\in\Qbar
\end{equation}
holds, where $X$ is a variable. If the functional equation \eqref{eq:SS} is satisfied, then by Lemma \ref{lem:S}
\begin{equation}\label{eq:SSS}
\prod_{j=1}^s(1-\beta_jX)^{d_j}=1
\end{equation}
holds. Taking the logarithmic derivative of \eqref{eq:SSS} and then multiplying both sides by $-X$, we get
$$\sum_{j=1}^s\beta_jd_j\frac{X}{1-\beta_jX}=0,$$
which is a special case of \eqref{eq:RRR} since $\beta_jd_j$ $(1\leq j\leq s)$ are not all zero. It is easily seen that \eqref{eq:RRR} does not hold since $\beta_j$ $(1\leq j\leq s)$ are nonzero distinct numbers and since $c_{jm}$ $(0\leq j\leq s,\ 0\leq m\leq M)$ are not all zero. Therefore neither the case (i) nor (ii) arises, which is a contradiction.
\end{proof}

\section*{Acknowledgements}
The author would like to thank his supervisor, Prof. Taka-aki Tanaka, for his valuable advice and guidance. The author also would like to thank Mr. Yusuke Tanuma for his helpful comments.

Department of Mathematics, Faculty of Science and Technology, Keio University 3-14-1 Hiyoshi, Kohoku-ku, Yokohama, 223-8522, Japan

{\it E-mail address}: {\tt haru1111@keio.jp}
\end{document}